\documentclass[preprint,10pt]{elsarticle}

\usepackage{amsthm,amsmath,amsfonts,amssymb,amscd,mathrsfs}
\usepackage{txfonts}
\usepackage{supertabular,soul}
\usepackage[usenames,dvipsnames]{xcolor}
\usepackage{tikz, graphicx,color,geometry}
\usepackage{multirow}
\usetikzlibrary{arrows}
\usepackage{blindtext,ulem}

\usepackage[pdftex,
            pdfauthor={Webb},
            pdftitle={Automorphisms, Equitable Partitions, and Spectral Graph Theory},
            pdfsubject={Equitable Partitions, Spectral Graph Theory},
            pdfkeywords={Equitable Partitions, Spectral Graph Theory}]{hyperref}

\usepackage{bbm} 
\usepackage{hyperref}
\usepackage{yfonts}
\usepackage{eucal}
\usepackage{overpic}
\usetikzlibrary{calc}
\usepackage{enumitem}

\newcommand{\comment}[1]{}

\def\m{\medskip}
\def\s{\smallskip}

\newcommand{\defital}{\textit}
\newcommand{\ds}{\displaystyle}

\newcommand{\Aut}{\text{Aut}}

\newcommand{\cT}{\mathcal{T}}
\newcommand{\tM}{\widetilde{M}}
\newcommand{\tA}{\tilde{A}}
\newcommand{\cU}{\mathcal{U}}
\renewcommand{\so}{\mathscr{O}}

\newtheorem{thm}{Theorem}[section]

\newtheorem{prop}[thm]{Proposition}

\newtheorem{cor}[thm]{Corollary}

\theoremstyle{definition}
\newtheorem{remark}[thm]{Remark}

\newtheorem{defn}[thm]{Definition}

\newtheorem{example}[thm]{Example}

\theoremstyle{remark}

\providecommand*{\propertyautorefname}{Property}

\let\oldmarginpar\marginpar
\renewcommand\marginpar[1]{\oldmarginpar[\raggedleft\footnotesize #1]%
{\raggedright\footnotesize #1}}

\renewcommand{\emph}{\textit}
\begin{document}
\begin{frontmatter}

\date{\today}

\title{Extensions and Applications of Equitable Decompositions for Graphs with Symmetries}
\author[amanda]{Amanda Francis}
\address[amanda]{Department of Mathematics, Engineering, and Computer Science, Carroll College, Helena, MT 59601, USA, afrancis@carroll.edu}
\author[dallas]{Dallas Smith}
\address[dallas]{Department of Mathematics, Brigham Young University, Provo, UT 84602, USA, dallas.smith@mathematics.byu.edu }
\author[derek]{\\Derek Sorensen}
\address[derek]{Mathematical Institute, University of Oxford, Oxford, England, derek.sorensen@maths.ox.ac.uk}
\author[ben]{Benjamin Webb}
\address[ben]{Department of Mathematics, Brigham Young University, Provo, UT 84602, USA, bwebb@mathematics.byu.edu}


\begin{abstract}
{We extend the theory of equitable decompositions introduced in \cite{BFW}, where it was shown that if a graph has a particular type of symmetry, i.e. a uniform or basic automorphism $\phi$, it is possible to use $\phi$ to decompose a matrix $M$ appropriately associated with the graph. The result is a number of strictly smaller matrices whose collective eigenvalues are the same as the eigenvalues of the original matrix $M$. We show here that a large class of automorphisms, which we refer to as \emph{separable},
 can be realized as a sequence of basic automorphisms, allowing us to equitably decompose $M$ over any such automorphism. We also show that not only can a matrix $M$ be decomposed but that the eigenvectors of $M$ can also be equitably decomposed. Additionally, we prove under mild conditions that if a matrix $M$ is equitably decomposed the resulting divisor matrix, which is the divisor matrix of the associated equitable partition, will have the same spectral radius as the original matrix $M$. Last, we describe how an equitable decomposition effects the Gershgorin region $\Gamma(M)$ of a matrix $M$, which can be used to localize the eigenvalues of $M$. We show that the Gershgorin region of an equitable decomposition of $M$ is contained in the Gershgorin region $\Gamma(M)$ of the original matrix. We demonstrate on a real-world network that by a sequence of equitable decompositions it is possible to significantly reduce the size of a matrix' Gershgorin region.}
\end{abstract}


\end{frontmatter}
\section{Introduction}

Spectral graph theory is the study of the relationship between two objects, a graph $G$ and an associated matrix $M$. The goal of this theory is to understand how spectral properties of the matrix $M$ can be used to infer structural properties of the graph $G$ and vice versa.

The particular structures we consider in this paper are graph symmetries. A graph is said to have a \emph{symmetry} if there is a permutation $\phi: V(G) \to V(G)$ of the graph's vertices $V(G)$ that preserves (weighted) adjacencies. The permutation $\phi$ is called an \emph{automorphism} of $G$, hence the symmetries of the graph $G$ are characterized by the graph's set of automorphisms. Intuitively, a graph automorphism describes how parts of a graph can be interchanged in a way that preserves the graph's overall structure.  In this sense these \emph{smaller parts}, i.e., subgraphs, are symmetrical and together these subgraphs constitute a graph symmetry.

In a previous paper \cite{BFW} it was shown that if a graph $G$ has a particular type of automorphism $\phi$ then it is possible to decompose any matrix $M$ that respects the structure of $G$ into a number of smaller matrices $M_{\phi},B_1,\dots,B_{k-1}$. Importantly, the eigenvalues of $M$ and the collective eigenvalues of these smaller matrices are the same, i.e.
\begin{equation*}\label{eq:first}
\sigma(M)=\sigma(M_{\phi})\cup\sigma(B_1)\cup\cdots\cup\sigma(B_{k-1}).
\end{equation*}

This method of decomposing a matrix into a number of smaller pieces over a graph symmetry is referred to as an \emph{equitable decomposition} due to its connection with the theory of equitable partitions. An \emph{equitable partition} of the adjacency matrix $A$ associated with a graph $G$ is a partition of the graph's set of vertices, which may arise from an automorphism $\phi$ of $G$, yielding a smaller matrix $A_{\phi}$ whose eigenvalues form a subset of the spectrum of $A$ (Theorem 9.3.3 of \cite{Godsil} , Theorem 3.9.5 of \cite{cvet}  ).

In \cite{BFW} the notion of an equitable partition is extended to other matrices beyond the adjacency matrix of a graph to include various Laplacian matrices, distance matrices, etc. (see Proposition 3.4).  This class of matrices, referred to as \emph{automorphism compatible} matrices, are those matrices associated with a graph $G$ that can be equitably decomposed over an automorphism $\phi$ of $G$. In particular, the matrix $M_{\phi}$ in the resulting decomposition is the same as the matrix that results from an equitable decomposition of $G$ if $M=A$ is the adjacency matrix of $G$.

The particular types of automorphisms considered in \cite{BFW} are referred to as uniform and basic automorphisms. A \emph{uniform} automorphism $\phi$ is one in which all orbits have the same cardinality (see Remark \ref{rem:orbits}). A \emph{basic} automorphism $\phi$ is an automorphism for which all \emph{nontrivial orbits}, i.e. orbits of size greater than one, have the same cardinality. Hence, any uniform automorphism is a basic automorphism.

Since many graph automorphisms are not basic, a natural question is whether an automorphism compatible matrix $M$ can be decomposed over {a nonbasic automorphism}. Here we show that if an automorphism is separable, i.e. is an automorphism whose order is the product of distinct primes, then there are basic automorphisms $\psi_0,\psi_1,\dots,\psi_h$ that induce a sequence of equitable decompositions on $M$. The result is a collection of smaller matrices $M_{\phi},B_1,\dots,B_{k}$ such that
\begin{equation*}
\sigma(M)=\sigma(M_{\phi})\cup\sigma(B_1)\cup\cdots\cup\sigma(B_{k})
\end{equation*}
where $k = p_0p_1 \ldots p_h -1$ and $M_{\phi}$ is again the matrix associated with the equitable partition induced by $\phi$ (see Theorem \ref{thm:2}). That is, the theory of equitable decompositions can be extended to any separable automorphism of a graph $G$.

We then show that not only can a matrix $M$ be equitably decomposed but also the eigenvectors (and generalized eigenvectors) of $M$ can be decomposed over any basic or separable automorphism $\phi$. More specifically, if $M$ can be decomposed into the matrices $M_{\phi},B_1,\dots,B_{k}$ over $\phi$ then the eigenvectors of $M$ can be explicitly constructed from the eigenvectors of $M_{\phi},B_1,\dots,B_{k}$. That is, the eigenvectors of these smaller matrices form the building blocks of the larger eigenvectors of the original matrix $M$ (see Theorem \ref{eigvec} for basic automorphisms), which we refer to as an equitable decomposition of the eigenvectors (and generalized eigenvectors) of $M$.

It is worth mentioning that if $\phi$ is any automorphism of $G$ then some power $\psi=\phi^\ell$ is a separable automorphism. Hence, if any automorphism of a graph is known it is possible to use this automorphism or some power of this automorphism to equitably decompose an associated matrix $M$.

Importantly,  an equitable decomposition of $M$, as opposed to its spectral decomposition, does not require any knowledge of the matrix' eigenvalues or eigenvectors. Only the knowledge of a symmetry of $G$ is needed. {In fact, if an automorphism describes a graph symmetry that involves only part of the graph i.e. a \emph{local symmetry},} this local information together with the theory presented here can be used to determine properties of the graph's associated eigenvalues and eigenvectors, which in general depend on the entire graph structure!

This method of using local symmetries to determine spectral properties of a graph is perhaps most useful in analyzing the spectral properties of real-world networks. One reason is that many networks have a high degree of symmetry \cite{MacArthur} when compared, for instance, to randomly generated graphs \cite{Aldous2000,Newman10,Strogatz03,Watts99}. From a practical point of view, the large size of these networks limit our ability to quickly compute their associated eigenvalues and eigenvectors. However, their high degree of symmetry suggests that {it may be possible to effectively estimate a network's spectral properties by equitably decomposing the network over local symmetries, which is a potentially much more feasible task.}

For instance, we show that in a network given by the graph G with automorphism compatible matrix $M$, the spectral radius of $M$ and its divisor matrix $M_{\phi}$ are equal if $M$ is both nonnegative and irreducible (see Proposition \ref{lem:3}). This result is of interest by itself since the spectral radius can be used to study stability properties of a network's dynamics \cite{BW10,BW11}.

Additionally, we show that the Gershgorin region associated with an equitable decomposition is contained in the Gershgorin region associated with the original undecomposed matrix (see Theorem \ref{thm:Gersh}). Since the eigenvalues of a matrix are contained in its Gershgorin region \cite{Gershgorin31}, then by equitably decomposing a matrix over some automorphism that is either basic or separable it is possible to gain improved eigenvalue estimates of the matrix' eigenvalues. Again, this result is potentially useful for estimating the eigenvalues associated with a real network {as such networks often have a high degree of symmetry.}

This paper is organized as follows. In Section \ref{sec:EP} we summarize the theory of equitable decompositions found in \cite{BFW}. In Section \ref{sec:3} we describe how the theory of equitable decompositions can be extended to separable automorphisms by showing that a decomposition over such an automorphism $\phi$ can be realized as a sequence of decompositions over basic automorphisms $\psi_0,\psi_1,\dots,\psi_h$ (Corollary \ref{cor:4}). We also present an algorithm describing how these automorphisms can be generated and used to equitably decompose an associated matrix.

In Section \ref{sec:4} we introduce the notion of an equitable decomposition of a matrix' eigenvectors and generalized eigenvectors (Theorem \ref{eigvec}).  We also demonstrate that $M$ and $M_{\phi}$ have the same spectral radius if $M$ is both nonnegative and irreducible (Proposition \ref{lem:3}).

In Section \ref{sec:5} we show that we gain improved eigenvalue estimates using Gershgorin's theorem when a matrix is equitably decomposed (Theorem \ref{thm:Gersh}), which we demonstrate on a large social network from the pre-American revolutionary war era. In Section \ref{sec:6} we show how the theory of equitable decompositions can be directly applied to graphs. Section \ref{sec:7} contains some closing remarks including a few open questions regarding equitable decompositions.

\section{Graph Symmetries and Equitable Decompositions}\label{sec:EP}

The main objects considered in this paper are graphs. A \emph{graph} $G$ is made up of a finite set of vertices $V(G)=\{1,\dots,n\}$ and and a finite set of edges $E(G)$. The vertices of a graph are typically represented by points in the plane and an edge by a line or curve in the plane that connects two vertices. A graph can be \emph{undirected}, meaning that each edge $\{i,j\}\in E$ can be thought of as an unordered set or a multiset if $i=j$ ($\{i,i\}\in E$). A graph is \emph{directed} when each edge is {\emph{directed}, in which case $(i,j)$ is an ordered tuple. In both a directed and undirected graph, a} \emph{loop} is an edge with only one vertex ($\{i, i\}\in E$). A \emph{weighted graph} is a graph, either directed or undirected, in which each edge $\{i,j\}$ or $(i,j)$ is assigned a numerical weight $w(i,j)$.

In practice there are a number of matrices that are often associated with a given graph $G$. Two of the most common are the adjacency matrix $A=A(G)$ and the Laplacian matrix $L=L(G)$ of a graph $G$. The adjacency matrix of a graph is the $0$-$1$ matrix given by
\[
A_{ij}=
\begin{cases}
1 &\text{if} \ \ (i,j)\in E(G)\\
0 &\text{otherwise}.
\end{cases}
\]
To define the Laplacian matrix of a \emph{simple graph} $G$, i.e. an unweighted undirected graph without loops, let $D_G=\text{diag}[\text{deg}(1),\dots,\text{deg}(n)]$ denote the degree matrix of $G$. Then the Laplacian matrix $L(G)$ is the matrix $L(G)=D_G-A(G)$. If $G$ is a weighted graph its \textit{weighted adjacency matrix} $W=W(G)$ is given by its edge weights $W_{ij} = w(i,j)$,  where  $w(i,j) \neq 0$ if and only if $(i,j) \in E(G)$.

For an $n\times n$ matrix $M=M(G)$ associated with a graph $G$ we let $\sigma(M)$ denote the \emph{eigenvalues} of $M$. For us $\sigma(M)$ is a multiset with each eigenvalue in $\sigma(M)$ listed according to its multiplicity.

One of our main concerns in this paper is understanding how symmetries in a graph's structure (i) affect the eigenvalues and eigenvectors of a matrix $M=M(G)$ and (ii) how these symmetries can be used to decompose the matrix $M$ into a number of smaller matrices in a way that preserves the eigenvalues of $M$. Such graph symmetries are formally described by the graph's set of automorphisms.

\begin{defn}
An \emph{automorphism} $\phi$ of an unweighted graph $G$ is a permutation of $V(G)$ such that the adjacency matrix $A=A(G)$ satisfies $A_{ij} = A_{\phi(i) \phi(j)}$ for each pair of vertices $i$ and $j$. Note that  this is equivalent to saying
$i$ and $j$ are adjacent in $G$ if and only if $\phi(i)$ and $\phi(j)$ are adjacent in $G$. For a weighted graph $G$, if $w(i,j) = w(\phi(i), \phi(j))$ for each pair of vertices $i$ and $j$, then $\phi$ is an automorphism of $G$.

The set of all automorphisms of $G$ is a group, denoted by $\Aut(G)$. The \emph{order} of $\phi$ is the smallest positive integer $\ell$ such that $\phi^\ell$ is the identity.
\end{defn}

\begin{remark}\label{rem:orbits}
For a graph $G$ with automorphism $\phi$, we define the relation $\sim$ on $V(G)$ by $u \sim v$ if and only if $v = \phi^j(u)$ for some nonnegative integer $j$. It follows that $\sim$ is an equivalence relation on $V(G)$, and the equivalence classes are called the \emph{orbits} of $\phi$. The orbit associated with the vertex $i$ is denoted $\so_\phi(i)$. 
\end{remark}
Here, as in \cite{BFW} we consider those matrices $M=M(G)$ associated with a graph $G$ whose structure mimics the symmetries of the graph.
\begin{defn}\label{def:autocomp}\textbf{(Automorphism Compatible)}
Let $G$ be a graph on $n$ vertices. An $n \times n$ matrix $M$ is \emph{automorphism compatible} on $G$ if, given any automorphism $\phi$ of $G$ and any $i, j \in \{1, 2, \ldots, n\}$,
$M_{\phi(i) \phi(j)} = M_{i j}$.
\end{defn}

Some of the most well-known matrices that are associated with a graph are automorphism compatible. This includes the adjacency matrix, combinatorial Laplacian matrix, signless Laplacian matrix, normalized Laplacian matrix, and distance matrix of a simple graph. Additionally, the weighted adjacency matrix of a weighted graph is automorphism compatible. (See Proposition 3.4, \cite{BFW}.)

If $M=M(G)$ is an automorphism compatible matrix, $M$ can be decomposed over an automorphism $\phi$ of $G$ into a number of smaller matrices if $\phi$ is a basic automorphism.

\begin{defn} \textbf{(Basic Automorphism)}
If $\phi$ is an automorphism of a graph $G$ with orbits of size $k >1$ and possibly 1, then $\phi$ is a \textit{basic automorphism} of $G$ with orbit size $k$. Any vertices with orbit size 1 are said to be \emph{fixed} by $\phi$.
\end{defn}

Given a basic automorphism $\phi$ with orbit size $k$, we form a set by choosing one vertex from each orbit of size $k$. We call this set $\cT_0$ of vertices a \textit{semi-transversal} of the orbits of $\phi$. Further we define the set
\begin{equation}\label{eq:tranversal}
\cT_\ell = \{\phi^\ell(v) \ | \ v \in \cT_0\}
\end{equation}
for $\ell = 0,1, \ldots, k-1$ to be the $\ell$th power of $\cT_0$ and we let $M[\mathcal{T}_i,\mathcal{T}_j]$ be the submatrix of $M$ whose rows are indexed by $\mathcal{T}_i$ and whose columns are indexed by $\mathcal{T}_j$. This notion of a semi-transversal allows us to decompose an automorphism compatible matrix $M=M(G)$ in the following way.

\begin{thm}\label{thm:2} \textbf{(Basic Equitable Decomposition)} \cite{BFW}
Let $G$ be a graph on $n$ vertices, let $\phi$ be a basic automorphism of $G$ of size $k>1$, let $\cT_0$ be a semi-transversal of the $k$-orbits of $\phi$, let $\cT_f$ be the vertices fixed by $\phi$, 
 and let $M$ be an automorphism compatible matrix on $G$.  Set $F = M[\cT_f,\cT_f]$, $H = M[\cT_f,\cT_0]$, $L=M[\cT_0,\cT_f]$, $M_m = M[\cT_0, \cT_m]$, for $m = 0, 1, \ldots, k-1$, $\omega = e^{2 \pi i /k}$, and
\begin{equation}\label{eq:B}
B_j = \sum_{m=0}^{k-1} \omega^{jm} M_m,  \ \ j = 0, 1, \ldots, k-1.
\end{equation}
Then there exists an invertible matrix $S$ that can be explicitly constructed such that
\begin{equation}\label{eq:spectrum2}
S^{-1}MS=M_{\phi}\oplus B_1\oplus B_2\oplus\cdots B_{k-1}
\end{equation}
where $M_\phi=\left[\begin{array}{rr} F & kH \\ L & B_0 \end{array}\right].$ Thus
$\sigma(M) = \sigma\left(M_\phi \right)
\cup \sigma(B_1) \cup \sigma(B_2) \cup \cdots \cup \sigma(B_{k-1})$.
\end{thm}

The decomposition in Equation \eqref{eq:spectrum2} is referred to as an \emph{equitable decomposition} of $M$ associated with the automorphism $\phi$. The reason for this is that this decomposition is related to an equitable partition of the graph $G$.

\begin{defn}\label{def:EP}\textbf{(Equitable Partition)}
An \emph{equitable partition} of a graph $G$ and a matrix $M$ associated with $G$, is a partition $\pi$ of $V(G)$, $V(G) = V_1 \cup \ldots \cup V_k$ which has the property that for all $i$, $j \in \{1, 2, \ldots, k\}$
\begin{equation}\label{eq:1}
\sum_{t \in V_j} M_{st} = D_{ij}
\end{equation}
is a constant $D_{ij}$ for any  $s \in V_i$. The $k \times k$ matrix $M_\pi = D
$ is called the \defital{divisor matrix} of $M$ associated with the partition $\pi$.
\end{defn}

Definition \ref{def:EP} is, in fact, an extension of the standard definition of an equitable partition, which is defined for \emph{simple graphs}. For such graphs the requirement that $\pi$ be an equitable partition is equivalent to the condition that any vertex $\ell \in V_i$ has the same number of neighbors in $V_j$ for all $i,j \in \{1, \ldots, k\}$ (for example, see p. 195-6 of  \cite{Godsil}).

An important fact noted in \cite{BFW} is that, if $\phi$ is a basic automorphism of $G$ and $M$ is an automorphism compatible matrix associated with $G$, the orbits of $\phi$ form an equitable partition of $V(G)$ (see Proposition 3.2, \cite{BFW}). If $M$ is equitably decomposed over the basic automorphism $\phi$ as in Equation \eqref{eq:spectrum2}, the matrix $M_\phi$ in the resulting decomposition is in fact the divisor matrix $D$ associated with the equitable partition induced by $\phi$ (see Theorem 4.4, \cite{BFW}), which is the reason this decomposition is referred to as an equitable decomposition.

If $\phi$ is an automorphism in which every orbit has the same size $k>1$ then $\phi$ is referred to as a \emph{uniform automorphism} of \emph{size} $k$. Any uniform automorphism is clearly a basic automorphism in the sense that it is a basic automorphism that fixes no vertices. Thus, Theorem \ref{thm:2} holds for uniform automorphisms as well, in which case the divisor matrix $M_{\phi}=B_0$.

If a graph $G$ has a non-basic automorphism $\phi$, the current theory of equitable decompositions does not directly allow us to decompose a matrix $M=M(G)$ over $\phi$. In the following section we show that an automorphism compatible matrix $M$ can be decomposed with respect to any separable automorphism $\phi$ of $G$ via a sequence of basic automorphisms.

\section{Equitable Partitions using Separable Automorphisms}\label{sec:3}

Many graph automorphisms are not basic automorphisms. In this section we will demonstrate how to equitably decompose a matrix with respect to an arbitrary separable automorphism by repeated use of Theorem \ref{thm:2}. Here, a \emph{separable automorphism} $\phi$ of a graph $G$ is an automorphism whose order $|\phi|=p_0p_1\dots p_h$ where $p_0,p_1,\dots,p_h$ are distinct primes. Before we can describe an equitable decomposition over a separable automorphism we first need the following propositions and algorithm.



\begin{remark}\label{rem:graphdecomp}
Notice that if {$B=M_{\phi}\oplus B_1\oplus \cdots\oplus B_{k-1}$} is the equitable decomposition of a matrix $M$ with respect to $\phi$, then we may view $B$ as the weighted adjacency matrix for a new graph $\tilde{G}$  with the same vertex set as $G$.
In the proofs of Theorems 3.8 and 4.4 of \cite{BFW} the rows and columns of the matrix $M$ are labeled in the order $\cU, \cT_0, \ldots, \cT_{k-1}$.  We continue this \emph{row/column labeling}, so that the labeling for the divisor matrix $M_\phi$ follows the ordering $\cU,\cT_0$, and for all remaining matrices $B_m$ in the decomposition the labeling follows $\cT_m$.
\end{remark}

\begin{prop}\label{prop:dd}
Let $\phi$ be an automorphism of order $pq$ with $p$ prime and $p \nmid q$ of 
a graph $G=(V,E,w)$ with automorphism compatible matrix $M$.  
Then $\psi = \phi^{q}$ is a basic automorphism of $G$ with order $p$ {and} we can construct an automorphism $\tilde{\phi}$ associated with the equitable decomposition of $M$ over $\psi$ of order $q$ such that the divisor matrix $(M_\psi)_{\tilde{\phi}} = M_\phi$.
\end{prop}

\begin{proof}
Let $M$  and $\phi$ be as described in Proposition \ref{prop:dd}.
For ease of notation, let $M(i,j)=M_{ij}$, the $ij^{th}$ element of $M$. Certainly, the automorphism $\psi=\phi^{q}$ must have order $p$ implying that $\psi$ is a basic automorphism.

In order to perform an equitable decomposition with respect to $\psi$, we choose a semi-transversal $\cT_0$ in the following way: For each orbit of $\phi$ that is not fixed by $\psi$, pick an element $a$.  Then $|a| = pq_a$ for some $q_a \in \mathbb{Z}_{>0}$. We add the elements of the set $ \{ a, \phi^{p}(a), \phi^{2p}(a), \ldots, \phi^{(q_a-1)p}(a) \}$ to $\cT_0$. We let $\mathcal{U}$ denote the set of vertices fixed by $\psi$. 

To see that $\cT_0$ is a semi-transversal, notice that the element $a$ gives $q_a$ orbits under $\psi$ and there are $q_a$ elements in the set listed above.  We now show that the elements in the above set must come from different orbits.  Suppose that $\phi^{\eta p}(a)$ and $\phi^{\eta' p}(a)$ (with $0< |\eta - \eta'| < q_a$) are in the same $\psi$-orbit, then for some integer $s<p$,
$\phi^{\eta p}(a) =\psi^s  \phi^{\eta' p}(a) = \phi^{q_as + \eta'p}(a)$. Thus, $q_a \ | \ (\eta - \eta')$, a contradiction.

Now we define a map $\tilde\phi = \phi^p$, and notice that $\tilde\phi (\cT_m) \subseteq \cT_m$.
Recall that the decomposed matrix $B= M_\psi \oplus B_1 \oplus \cdots \oplus B_{p-1}$ guaranteed by Theorem 4.4 of \cite{BFW}, will have row and column order agreeing with the vertex order $\cU, \cT_0, \cT_1, \ldots, \cT_{p-1}$. 
Thus, to show that $\tilde\phi$ is an automorphism of $B$, we need only demonstrate that each $\tilde\phi|_{\cT_m}$ is an automorphism on $B_m$ (see Theorem \ref{thm:2}).
Recall that
\[
B_m = \sum_{j=0}^{p-1} \omega^{mj} M[\cT_0,\cT_j],
\]
Thus, if $a, b \in \cT_m$, and we wish to calculate the $(a,b)$ entry in $B_m$, we must examine the corresponding entries in $M$ which come from $\cT_0$ and $\cT_j$. This is expressed in the first equality below:
\begin{align*}
B_m(\tilde\phi|_{\cT_m}(a),\tilde\phi|_{\cT_m}(b))=\sum_{j=0}^{p-1}\omega^{mj}
M(\psi^{-m}\tilde\phi(a), \psi^{j-m}\tilde\phi(b))
&=\sum_{j=0}^{p-1}\omega^{mj} M(\phi^{p}\psi^{-m}(a), \phi^{p}\psi^{j-m}(b))\\
=\sum_{j=0}^{p-1}\omega^{mj} M(\psi^{-m}(a), \psi^{j-m}(b))&= B_m(a,b)
\end{align*}
where the second equality holds because $\psi = \phi^q$, and the third equality is the defining property of automorphism compatible matrices. Thus,  $\tilde\phi$ is an automorphism on each $B_m$ and subsequently, $\tilde{\phi} $ is an automorphism on $B$, the decomposition of $M$.

The equality below similarly shows that $\tilde \phi$ is an automorphism of the vertices $(a,b)$ that appear in $M_\psi$ (those in $\cU \cup \cT_0$).
\[
M_\psi(\tilde\phi(a), \tilde\phi(b)) = M_\psi(\phi^p (a),\phi^p(b)) =
\sum_{m = 1}^{|\so_\psi(b)|} M(\phi^p(a), \phi^{qm + p}(b)) =
\sum_{m = 1}^{|\so_\psi(b)|} M((a), \phi^{qm}(b)) = M_\psi(a,b).
\]

To show the final equality in Proposition \ref{prop:dd}, we consider a semi-transversal $\widetilde{\cT}_0$ of $\tilde \phi$ and the set of vertices fixed by $\tilde \phi$ (which we call $\widetilde \cU$).  Then for vertices $a$ and $b$ in $\widetilde\cU \cup \widetilde{\cT}_0$ we find that 

$$\begin{array}{l}
\ds{M_\phi (a,b)=\sum_{s\in \so_\phi(b)}M(a,s)=\sum_{m\in \so_{\tilde{\phi}}(b)}\left( \sum_{s\in\so_\psi(m)}M(a,s)\right)} \ds{ \ =\sum_{s\in\so_{\tilde{\phi}}(b)} M_{\psi}(a,s)=(M_{\psi})_{\tilde{\phi}}(a,b)}.
\end{array}$$

The second equality holds because $$\so_\phi(s)=\bigcup_{t\in\so_{\tilde{\phi}}(S)} \so_\psi(t)$$ 
Hence, $M_{\phi}=(M_{\psi})_{\tilde{\phi}}$.

\end{proof}

Thus, for any automorphism $\phi\in Aut(G)$ of order $pq = \ell$ where $p$ is prime and $p \nmid q$,  we can equitably decompose an automorphism compatible matrix $M$ over  $\phi^{q}$ and subsequently create another automorphism $\tilde{\phi}$ associated with the decomposed matrix. In fact, if $\phi$ is separable then we can repeat this process until we exhaust each of the distinct prime factors $p_0,p_1,\dots,p_h$ of $\ell$ where $\ell=p_0p_1\cdots p_h$ is the order of $\phi$. This decomposition of the matrix $M$, is summarized in the following proposition.

\begin{thm}\label{thm:3} \textbf{(Equitable Decompositions Over Separable Automorphisms)}
Let $\phi$ be any separable automorphism of a graph $G$ with automorphism compatible matrix $M$. Then there exists basic automorphisms $\psi_1,\dots,\psi_{h}$ that induce a sequence of equitable decompositions on $M$, such that the divisor matrix
\[
M_\phi = (\dots(M_{\psi_0})_{\psi_1} \ldots)_{\psi_h}.
\]
\end{thm}

\begin{proof}
This follows from repeated use of Proposition \ref{prop:dd}.
\end{proof}


We now give an algorithm for decomposing a graph with respect to any of its automorphisms.

\begin{center}
\emph{Performing Equitable Decompositions using Separable Automorphisms}
\end{center}

\begingroup\raggedright\leftskip=20pt\rightskip=20pt

For a graph $G$ with automorphism compatible matrix $M$ and separable automorphism $\phi$ of order $\ell $ with prime factorization ${\ell=p_0p_1 \cdots p_h}$, set $M(0)  = M$, $\ell_0 = \ell$, and $\phi_0 = \phi$.  We perform $h+1$ sequential decompositions of $M$, one for each prime in our factorization.  Thus we will run through Steps a-c $h+1$ times to fully decompose the matrix.  
To begin we start with $i = 0$, and move to \emph{Step a.}\m

\vspace{0.1in}


\noindent\emph{\textbf{Step a:}} \emph{Let $\ell_{i+1} = \ell_i/p_i$.  Form the basic automorphism $\psi_{i} = \phi_i^{\ell_{i+1}}$ of order $p_i$}.\\

\noindent\emph{\textbf{Step b:}} \emph{Perform an equitable decomposition of $M(i)$ over $\psi_i$ as in Theorem \ref{thm:2} by choosing a semi-transversal $\mathcal{T}_0$ of the $p_i$-orbits of $\psi_i$ according the the method set out in Proposition \ref{prop:dd} and setting $\mathcal{U}$ to be the set of all {vertices} fixed by $\psi_i$. Let $\tM(i)$ be the matrix obtained from $M(i)$ by permuting the rows and columns to agree with the new vertex order: $\cU, \cT_0, \cT_1, \ldots, \cT_{p_i-1}$. Then define 
\[M(i+1)=S\tM(i)S^{-1}=\tM(i)_{\psi_i}\oplus B(i)_1 \oplus B(i)_2\oplus\dots\oplus B(i)_{p_i-1}\]}


\noindent\emph{\textbf{Step c:}} \emph{Define $\phi_{i+1}=\tilde{\phi}_i =(\phi_i)^{p_i}$ as described in the proof of Proposition \ref{prop:dd}. If $i < h$, then set $i = i+1$ and return to \emph{Step a}.  Otherwise, the decomposition is complete.}



\par\endgroup

\vspace{0.1in}

\begin{remark}
Each occurrence of {\it Step b} requires choosing a semi-transversal $\cT_0$ and setting up a new fixed set $\cU$ (determined by $\psi_i$).  By slight abuse of notation we will simply reuse the same notation for each `round', forgetting the previously used semi-transversals and fixed vertex sets. 
\end{remark}

\noindent The procedure described in Steps $a$--$c$ allows one to sequentially decompose a matrix $M$ over any of its separable automorphisms. By extension we refer to the resulting matrix as an \emph{equitable decomposition} of $M$ over $\phi$. The following example illustrates an equitable decomposition over a separable automorphism that is not basic.

\begin{figure}
\[
\raisebox{-23mm}{\begin{overpic}[scale=.2]{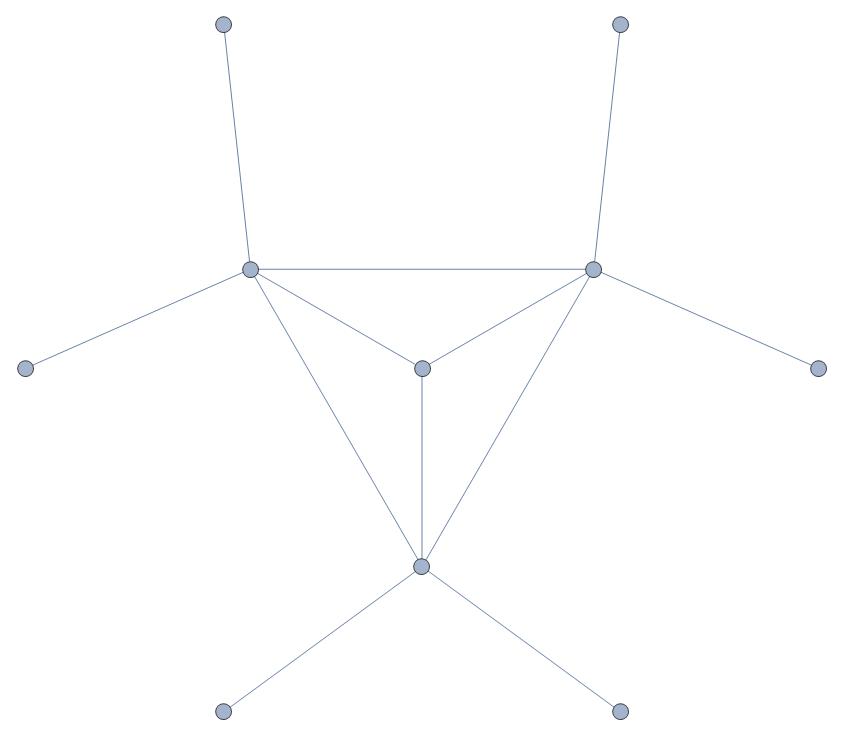}
    \put(52.5,40){\scriptsize$1$}
    \put(26.5,55.5){\scriptsize$5$}
	\put(2.75,45){\scriptsize$7$}
    \put(26,86){\scriptsize$6$}
    \put(72,55.5){\scriptsize$8$}
    \put(73,86){\scriptsize$10$}
    \put(96.5,45){\scriptsize$9$}
    \put(53,19.5){\scriptsize$2$}
    \put(72.5,-2){\scriptsize$4$}
    \put(25,-2){\scriptsize$3$}
    \end{overpic}}
\qquad
\begin {aligned}
A = \left[ {\begin{matrix}
   0 & 1 & 0 & 0 & 1 & 0 & 0 & 1 & 0 & 0  \\
   1 & 0 & 1 & 1 & 1 & 0 & 0 & 1 & 0 & 0  \\
   0 & 1 & 0 & 0 & 0 & 0 & 0 & 0 & 0 & 0  \\
   0 & 1 & 0 & 0 & 0 & 0 & 0 & 0 & 0 & 0  \\
   1 & 1 & 0 & 0 & 0 & 1 & 1 & 1 & 0 & 0  \\
   0 & 0 & 0 & 0 & 1 & 0 & 0 & 0 & 0 & 0  \\
   0 & 0 & 0 & 0 & 1 & 0 & 0 & 0 & 0 & 0  \\
   1 & 1 & 0 & 0 & 1 & 0 & 0 & 0 & 1 & 1  \\
   0 & 0 & 0 & 0 & 0 & 0 & 0 & 1 & 0 & 0  \\
   0 & 0 & 0 & 0 & 0 & 0 & 0 & 1 & 0 & 0  \\
 \end{matrix} } \right]
\end {aligned}
\]
\caption{The graph $G$ considered in Example \ref{ex:sequence} with automorphism $\phi=(2,5,8)(3,6,9,4,7,10)$ and adjacency matrix $A=A(G)$.}\label{fig:ex1}
\end{figure}
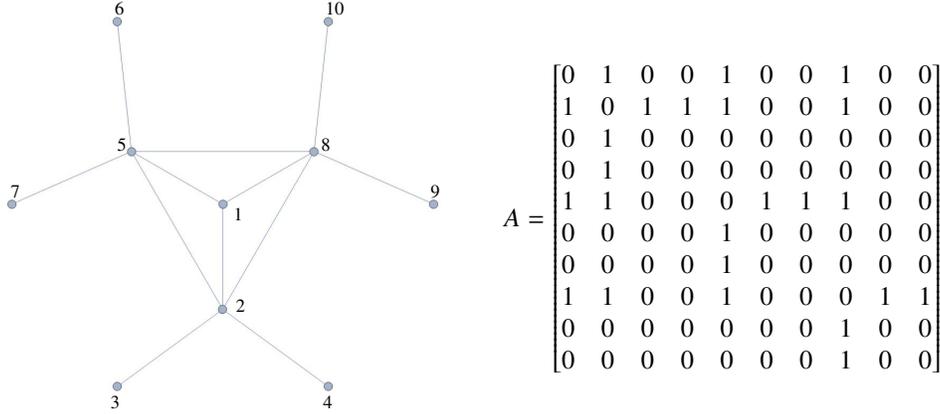

\begin{example}\label{ex:sequence}
Consider the graph $G$ in Figure \ref{fig:ex1} whose adjacency matrix $A=A(G)$ is also shown, which has the separable automorphism \begin{equation}\label{eq:triangles}
\phi = (2,5,8)(3,6,9,4,7,10).
\end{equation}
The automorphism $\phi$ has order $\ell=6 = 3\cdot2$. Since $\ell$ factors into two primes we will proceed through \emph{Steps a-c} two times to equitably decompose the adjacency matrix $A$ with respect to $\phi$. \s

\noindent \textbf{Round 1:}
Let $A(0) = A$, $\phi_0 = \phi$, $\ell_0 = 6$, and $p_0=3$.\s

\noindent \emph{Step a:} Note that $\ell_1 = 2$ and $\psi_0 =\phi_0^2=  (2,8,5)(3,9,7)(4,10,6)$, which  is a basic automorphism of order $p_0 = 3$.  \s

\noindent \emph{Step b:} To choose a semi-transversal $\mathcal{T}_0$ of $\psi_0$, we select vertices 2 and 3 from the orbits of $\phi_0$, and add $\phi_0^3(3) = 4$ to $\cT_0$ as well (following the method for choosing semi-transversals set out in Proposition \ref{prop:dd}.  Then $\mathcal{T}_1 = \{8,9,10\}$ and $\mathcal{T}_2 = \{5,7,6\}$, with $\mathcal{U} = \{1\}$. We let $\tA$ be the matrix obtained from $A$ by permuting the rows and columns to agree with the vertex order $\cU, \cT_0, \cT1, \cT_2 = 1, 2, 3, 4, 8, 9, 10, 5, 7, 6$ (in this case $\tA = A$).  Using Theorem \ref{thm:2}, we have
\begin{align*}
&\tA(0)_0=\tA\left[ \mathcal{T}_0,\mathcal{T}_0 \right] =\left[ \begin{matrix}
0&1&1 \\
1&0&0\\
1&0&0\end{matrix}\right], \ \tA(0)_1=\tA\left[ \mathcal{T}_0,\mathcal{T}_1 \right]=\left[ \begin{matrix} 1&0&0 \\ 0&0&0\\ 0&0&0\end{matrix}\right]=\tA(0)_2, \\ & \text{and} \
F(0)=\tA(0)\left[\mathcal{T}_f,\mathcal{T}_f \right]=\left[ \begin{matrix} 0 \end{matrix}\right], \ H(0)=L(0)^T=\left[ \begin{matrix} 1 & 0 & 0 \end{matrix}\right]
\end{align*}
from which the matrices
\begin{equation}\label{eq:7}
\tA(0)_{\psi_0}=\left[
\begin{matrix} F(0) & p_0H(0)   \\
   L(0) & B(0)_0
\end{matrix}\right]=\left[
\begin{matrix} 0 & 3 & 0 & 0  \\
   1 & 2 & 1 & 1  \\
   0 & 1 & 0 & 0   \\
   0 & 1 & 0 & 0
\end{matrix}\right], \
B(0)_1=B(0)_2=\left[
\begin{matrix}
   {-1} & 1 & 1   \\
      1 & 0 & 0  \\
      1 & 0 & 0
\end{matrix}\right] \ \text{and} \
A(1)=\left[
\begin{matrix}
   \tA(0)_{\psi_0} & 0 & 0   \\
      0 & B(0)_1 & 0  \\
      0 & 0 & B(0)_2
\end{matrix}\right]
\end{equation}
can be constructed.\s



\noindent \emph{Step c:} Next, we derive $\phi_1 = \tilde\phi_0 = (\phi_0)^{p_0} = (3,4) (6,7)(9,10)$. And notice that
\begin{gather}
\phi_1|_{\cT_0}=(3,4),~~~\phi_1|_{\cT_1}=(9,10),~~~\text{and}~~~\phi_1|_{\cT_2}=(6,7) \nonumber.
\end{gather}
where  $\phi_1$ is the automorphism associated with $A(1)$ guaranteed by Proposition \ref{prop:dd}. Since $\ell$ factors into two primes we proceed to Round 2.\s

\noindent \textbf{Round 2:} $A(1)$ and $\phi_1$ have been computed, $\ell_1 = 2$, and $p_1=2$.\s

\noindent \emph{Step a:} Since $\ell_2 = 1$ then $\psi_1 = \phi_1=(3,4)(6,7)(9,10)$, which is a basic automorphism of order $p_1 = 2$.  \s

\noindent \emph{Step b:} We choose the semi-transversal $\mathcal{T}_0 = \{3,6,9\}$ which causes $\mathcal{T}_1 = \{4,7,10\}$. Note that the set of fixed points is $\mathcal{U} = \{1,2,5,8\}$. 
Now, we create the matrix $\tA(1)$ from $A(1)$ by reordering the rows and columns to agree with the order $\cU, \cT_0, \cT_1 = 1,2,5,8, 3, 6, 9, 4, 7, 10$. 
By decomposing the matrix $\tA(1)$ as is Theorem \ref{thm:2} we have \s
\begin{align*}
&\tA(1)_0=\tA(0)\left[ \mathcal{T}_0,\mathcal{T}_0 \right] =\left[ \begin{matrix}
0&0&0 \\
0&0&0\\
0&0&0\end{matrix}\right], \
\tA(1)_1=\tA(0)\left[ \mathcal{T}_0,\mathcal{T}_1 \right]=\left[ \begin{matrix}
0&0&0 \\
0&0&0\\
0&0&0\end{matrix}\right], \\
&F(1)=\tA(0)\left[\mathcal{T}_f,\mathcal{T}_f \right]=\left[ \begin{matrix}
0&3&0&0 \\
1&2&0&0\\
0&0&-1&0 \\
0&0&0&-1 \end{matrix}\right], \text{and} \ \
H(1)=L(1)^T=\tA(0) [\cT_0, \cT_f] = \left[ \begin{matrix}
0&0&0 \\
1&0&0\\
0&1&0 \\
0&0&1 \end{matrix}\right]
\end{align*}
from which we can construct the matrices
\[B(1)_0=\left[ \begin{matrix}
0&0&0 \\
0&0&0\\
0&0&0\end{matrix}\right], \ B(1)_1=\left[ \begin{matrix}
0&0&0 \\
0&0&0\\
0&0&0\end{matrix}\right] \ \text{and} \
A(2)=\left[ \begin{matrix}
F(1)&p_1H(1)&0 \\
L(1)&B(1)_0&0\\
0&0&B(1)_1\end{matrix}\right].
\]

\noindent \emph{Step c:} The matrix $A(2)$ is now decomposed into blocks, and in this final step there is no need to find $\phi_2$ since the decomposition is complete.\\

Thus, our final decomposition is the matrix $A(2)$.  To see the block diagonal form of $A(2)$ we permute the rows and columns with a permutation matrix $P$ to put the associated vertices back in the original order. The result is the matrix
$$PA(2)P^{-1}= P\left[ \begin{array}{cccc|ccc|ccc}
   0 & 3 & 0 & 0 & 0 & 0 & 0 & 0 & 0 & 0  \\
   1 & 2 & 0 & 0 & 2 & 0 & 0 & 0 & 0 & 0  \\
   0 & 0 & -1& 0 & 0 & 2 & 0 & 0 & 0 & 0  \\
   0 & 0 & 0 & -1& 0 & 0 & 2 & 0 & 0 & 0  \\
   \hline
   0 & 1 & 0 & 0 & 0 & 0 & 0 & 0 & 0 & 0  \\
   0 & 0 & 1 & 0 & 0 & 0 & 0 & 0 & 0 & 0  \\
   0 & 0 & 0 & 1 & 0 & 0 & 0 & 0 & 0 & 0  \\
   \hline
   0 & 0 & 0 & 0 & 0 & 0 & 0 & 0 & 0 & 0  \\
   0 & 0 & 0 & 0 & 0 & 0 & 0 & 0 & 0 & 0  \\
   0 & 0 & 0 & 0 & 0 & 0 & 0 & 0 & 0 & 0  \\
   \end{array} \right]P^{-1}= \left[ \begin{array}{ccc|c|cc|c|cc|c}
   0 & 3 & 0 & 0 & 0 & 0 & 0 & 0 & 0 & 0  \\
   1 & 2 & 2 & 0 & 0 & 0 & 0 & 0 & 0 & 0  \\
   0 & 1 & 0 & 0 & 0 & 0 & 0 & 0 & 0 & 0  \\
   \hline
   0 & 0 & 0 & 0 & 0 & 0 & 0 & 0 & 0 & 0  \\
   \hline
   0 & 0 & 0 & 0 & -1& 2 & 0 & 0 & 0 & 0  \\
   0 & 0 & 0 & 0 & 1 & 0 & 0 & 0 & 0 & 0  \\
   \hline
   0 & 0 & 0 & 0 & 0 & 0 & 0 & 0 & 0 & 0  \\
   \hline
   0 & 0 & 0 & 0 & 0 & 0 & 0 & -1& 2 & 0  \\
   0 & 0 & 0 & 0 & 0 & 0 & 0 & 1 & 0 & 0  \\
   \hline
   0 & 0 & 0 & 0 & 0 & 0 & 0 & 0 & 0 & 0  \\
   \end{array} \right] $$

We can see in the final decomposition, that the (twice) decomposed divisor matrix is found in the first block $(A_{\psi_0})_{\psi_1}=\left[ \begin{matrix}
   0 & 3 & 0 \\
   1 & 2 & 2 \\
   0 & 1 & 0 \\
   \end{matrix} \right] $, which is the divisor matrix $A_\phi$ associated with the equitable partition of $A$ induced by $\phi$.
\end{example}

Before finishing this section we note that we have now stated two theorems regarding equitable decompositions of a graph over two types of automorphisms. In Theorem 4.4 of \cite{BFW} these were basic automorphisms and here in Theorem \ref{thm:3} we stated, and later showed, how a graph could be decomposed over any of its separable automorphisms. For the sake of unifying the theory of equitable decompositions we give the following corollary of these two theorems.

\begin{cor}\label{cor:4}
Let $G$ be a graph with automorphism compatible matrix $M$. If $\phi\in Aut(G)$ is either a basic or separable automorphism then there exists an invertible matrix $Q$ that can be explicitly constructed such that
\begin{equation*}
Q^{-1} M Q =  M_\phi \oplus B_1 \oplus \ldots \oplus B_{k}
\end{equation*}
for some $k\geq 1$. Hence, $\sigma(M)=\sigma(B_0)\cup\ldots\cup\sigma(B_{h})$. 
\end{cor}

The term \emph{explicitly constructed} in this corollary does not mean that there is only one such $Q$ that can be used to decompose the matrix $M$. 
Rather, given an automorphism $\phi$ on $G$ we can construct $Q$ knowing only  $\phi$ by once we've chosen an ordering on the primes in the order of $\phi$,  and semi-transversals at each step, as has been demonstrated throughout this section of the paper.

\begin{remark}
If $\phi$ is any automorphism of a graph $G$ whose order has prime decomposition
\[
|\phi|=p_0^{N_0}p_1^{N_1}\dots p_h^{N_h},
\]
then $\psi=\phi^\ell$ for $\ell=p_0^{N_0-1}p_1^{N_1-1}\dots p_h^{N_h-1}$ is a separable automorphism of $G$ with order $|\psi|=p_0p_1\dots p_h$. Hence, any automorphism compatible matrix $M$ of $G$ can be equitably decomposed with respect to $\psi$. Consequently, knowledge of any automorphism of a $G$ can be used to equitably decompose the matrix $M$. In this case, we decompose over $\psi$, not the original automorphism $\phi$.  
\end{remark}

\section{Eigenvectors and Spectral Radii Under Equitable Decompositions}\label{sec:4}

 Recall that for basic equitable decompositions, as given in \cite{BFW}, it is necessary to use an automorphism whose nontrivial cycles must all have the same length (not necessarily prime).  For this reason, throughout the rest of the paper we will use $k$ instead of $p$ to denote the common, nontrivial cycle length of our automorphisms. 
 
 The theory of equitable decompositions presented in \cite{BFW} and in the previous section not only allows us to decompose a matrix $M$ over an associated graph symmetry but can also be used to decompose the eigenvectors of $M$.That is, if $M_{\phi}\oplus B_1\oplus\cdots\oplus B_{{k-1}}$ is an equitable decomposition of $M$ over some basic or separable $\phi\in Aut(G)$ then the eigenvectors of $M$ can be explicitly constructed from the eigenvectors of $M_{\phi},B_1,\dots,B_{{k-1}}$. {This same theory can be used to show that} the spectral radius of $M$ and its divisor matrix $M_{\phi}$ are equal if $M$ is both nonnegative and irreducible (see Proposition \ref{lem:3}).

Our first result in this section introduces the notion of an equitable decomposition of the eigenvectors and generalized eigenvectors of a matrix $M$ associated with a graph $G$.

\begin{thm}\label{eigvec}\textbf{(Eigenvector Decomposition)}
Let $M$ be an $n\times n$ automorphism compatible matrix of the graph $G$. For $\phi$ a basic automorphism of $G$ with $N$ orbits of size $1$ and all other orbits of size $k>1$ let $M_\phi \oplus B_1\oplus \cdots \oplus B_{k-1}$ be an equitable decomposition of $M$. For $r=(n-N)/k$ suppose $\{\mathbf{u}_{m,\ell}:1\leq \ell \leq r\}$  is a (generalized) eigenbasis for $B_m$  for $1\leq m\leq k-1$ and $\{{{\mathbf{u}}_{0,i}}:1\leq i\leq N+r\}$ is a (generalized) eigenbasis for $M_{\phi}$ and where each ${\mathbf{u}}_{0,i}= {{\mathbf{w}}_i} \oplus {{\mathbf{v}}_i}$ with ${{\mathbf{w}}_i}\in\mathbb{C}^N\text{ and }{{\mathbf{v}}_i}\in\mathbb{C}^r$. Then a (generalized) eigenbasis of $M$ is the set
\begin{equation}\label{eq:vect}
\left\{{{\bf{0}}_N}\, \oplus {\left[ {\mathop  \bigoplus \limits_{j = 0}^{k - 1} {\omega ^{m j}}{{\bf{u}}_{m,\ell}}} \right] ,{{\bf{w}}_i}\,\oplus \left[   {\mathop  \bigoplus \limits_{j = 0}^{k - 1} {{\bf{v}}_i}} \right]\,:1 \leq m \leq k - 1,1 \leq \ell \leq r,1 \leq i \leq N + r,\omega=e^{2\pi i/k}\,} \right\}.
\end{equation}
Moreover, if $\mathbf{x}_{m,l}={{\bf{0}}_N}\oplus \left[ {\mathop  \bigoplus \limits_{j = 0}^{k - 1} {\omega ^{mj}}{{\bf{u}}_{m,l}}} \right]$ and ${{\bf{x}}_{0,i}} =  {{\bf{w}}_i} \oplus \left[{\mathop  \bigoplus \limits_{j = 0}^{k - 1} {{\bf{v}}_i}} \right]$ then the following hold.\\
(i) If $\lambda_{m,\ell}$ is the $\ell^{th}$ eigenvalue of $B_m$ then $\lambda_{m,\ell}\in\sigma(M)$ corresponds to the (generalized) eigenvector $\mathbf{x}_{m,l}$.\\
(ii) If $\lambda_{0,i}$ is an eigenvalue of $M_\phi$ then $\lambda_{0,i}\in\sigma(M)$ corresponds to the (generalized) eigenvector ${{\bf{x}}_{0,i}}$.
\end{thm}

\begin{proof}
Using the notation from \cite{BFW}, we use the basic automorphism to label the vertices of $G$ so that $M$ has the form
\begin{equation}\label{eq:circulant}
M = \left[\begin{array}{llllll}
F & H & H & H & \cdots & H \\
L & M_0 & M_1 & M_2 & \cdots & M_{k-1} \\
L & M_{k-1} & M_0 & M_1 & \cdots & M_{k-2} \\
L & M_{k-2} & M_{k-1} & M_0 & \cdots & M_{k-3} \\
\vdots & \vdots & \vdots & \vdots & & \vdots \\
L & M_1 & M_2 & M_3 & \cdots & M_0 \\
\end{array}\right],
\end{equation}
where $F$ is $N \times N$, $H$ is $N \times r$, $L$ is $r \times N$, and each $M_j$ is $r \times r$. Let $S = I_N \oplus R$ where
$$R = \left[ {\begin{matrix}
   I & I & I &  \cdots  & I  \\
   I & {\omega I} & {{\omega ^2}I} &  \cdots  & {{\omega ^{k - 1}}I}  \\
   I & {{\omega ^2}I} & {{\omega ^4}I} &  \ddots  & {{\omega ^{2(k - 1)}}I}  \\
    \vdots  &  \vdots  &  \ddots  &  \ddots  &  \vdots   \\
   I & {{\omega ^{k - 1}}I} & {{\omega ^{2(k - 1)}}I} &  \cdots  & {{\omega ^{{{(k - 1)}^2}}}I}  \\
 \end{matrix} } \right]$$
and $\omega$ is the $k^{th}$ root of unity. Now according to Theorem \ref{thm:2}, we can decompose $M$ as
\[
S^{-1} M S =
M_\phi\oplus B_1 \oplus B_2 \oplus \cdots \oplus B_{k-1}=B,
\]
\\
where $M_\phi=\left[\begin{array}{rr} F & kH \\ L & B_0 \end{array}\right]$.
Now let $\textbf{u}$ be a (generalized) eigenvector of $B$ corresponding to the eigenvalue $\lambda$, so that $\left( B - \lambda I \right)^t{\bf{u}} = \mathbf{0}$ for some positive integer $t$, where $t>1$ if $\textbf{u}$ is a generalized eigenvector.  We now consider the vector $S{\bf{u}}$ which has the property that
\begin{align*}
  \left( {M - \lambda I} \right)^t S{\bf{u}} & =  (SBS^{-1} - \lambda I)^t S{\bf{u}} \cr
  & =  (S(B - \lambda I)S^{-1})^t S{\bf{u}} \cr
    & =  (S(B - \lambda I)^t S^{-1})S{\bf{u}} \cr
     & =  S(B - \lambda I)^t {\bf{u}} \cr
  & = 0.
  \end{align*}
Thus, $S{\bf{u}}$ is a (generalized) eigenvector for $M$.

Because $B$ is block diagonal, the (generalized) eigenvectors of B are either $(\mathbf{0}_N^T \ \mathbf{0}_r^T \dots \mathbf{u}_{m,\ell}^T \dots \mathbf{0}_r^T )^T$ or $(\mathbf{w}_i^T \ \mathbf{v}_i^T \ \mathbf{0}_{r}^T \dots \mathbf{0}_r^T )^T$
where ${{\mathbf{u}}_{m,\ell}}$ is the $m^{th}$ component in this block vector and represents the $\ell^{th}$ (generalized) eigenvector of $B_m$ associated with eigenvalue $\lambda_{m,\ell}$ and ${\mathbf{w}_i}\oplus {\mathbf{v}}_i={\mathbf{u}}_{0,i}$, (${{\mathbf{w}}_i}\in\mathbb{C}^N\text{ and }{{\mathbf{v}}_i}\in\mathbb{C}^r$), the $i^{th}$ (generalized) eigenvector of $M_\phi$. Thus the (generalized) eigenvectors of $M$ are represented by $$S\left( {\begin{matrix}
   {{{\mathbf{{0}}}_N}}  \\
   {{{\mathbf{{0}}}_r}}  \\
    \vdots   \\
   {{{\mathbf{{u}}}_{m,l}}}  \\
    \vdots   \\
   {{{\mathbf{0}}_r}}  \\

 \end{matrix} } \right)
  = \left[ {\begin{matrix}
     I_N & 0 & 0 & 0 &  \cdots  & 0  \\
   0 & I & I & I &  \cdots  & I  \\
  0 & I & {\omega I} & {{\omega ^2}I} &  \cdots  & {{\omega ^{k - 1}}I}  \\
  0 & I & {{\omega ^2}I} & {{\omega ^4}I} &  \ddots  & {{\omega ^{2(k - 1)}}I}  \\
  \vdots &  \vdots  &  \vdots  &  \ddots  &  \ddots  &  \vdots   \\
  0 &  I & {{\omega ^{k - 1}}I} & {{\omega ^{2(k - 1)}}I} &  \cdots  & {{\omega ^{{{(k - 1)}^2}}}I}  \\

 \end{matrix} } \right]\left( {\begin{matrix}
   {{{\mathbf{0}}_N}}  \\
   {{{\mathbf{0}}_r}}  \\
    \vdots   \\
   {{{\mathbf{u}}_{m,\ell}}}  \\
    \vdots   \\
   {{{\mathbf{0}}_r}}  \\

 \end{matrix} } \right)= \left( {\begin{matrix}
    {{{\mathbf{0}}_N}}  \\
   {{{\mathbf{u}}_{m,\ell}}}  \\
   {{\omega ^m}{{\bf{u}}_{m,\ell}}}  \\
   {{\omega ^{2m}}{{\bf{u}}_{m,\ell}}}  \\
    \vdots   \\
   {{\omega ^{m(k - 1)}}{{\bf{u}}_{m,\ell}}}  \\

 \end{matrix} } \right) = {{{\bf{0}}_N}}\oplus \mathop  \bigoplus \limits_{j = 0}^{k - 1} {\omega ^{mj}}{{\bf{u}}_{m,l}}$$
 and
 $$S \left( {\begin{matrix}
    {\bf{w}}_{i}\\
   {{{\bf{v}}_{i}}}  \\
   {{{\bf{0}}_r}}  \\
    \vdots   \\
    \vdots   \\
   {{{\bf{0}}_r}}  \\

 \end{matrix} } \right)
  = \left[ {\begin{matrix}
     I_N & 0 & 0 & 0 &  \cdots  & 0  \\
   0 & I & I & I &  \cdots  & I  \\
  0 & I & {\omega I} & {{\omega ^2}I} &  \cdots  & {{\omega ^{k - 1}}I}  \\
  0 & I & {{\omega ^2}I} & {{\omega ^4}I} &  \ddots  & {{\omega ^{2(k - 1)}}I}  \\
  \vdots &  \vdots  &  \vdots  &  \ddots  &  \ddots  &  \vdots   \\
  0 &  I & {{\omega ^{k - 1}}I} & {{\omega ^{2(k - 1)}}I} &  \cdots  & {{\omega ^{{{(k - 1)}^2}}}I}  \\

 \end{matrix} } \right]\left( {\begin{matrix}
    {{{\bf{w}}_{i}}}  \\
    {{{\bf{v}}_{i}}}  \\
    {{{\bf{0}}_r}}  \\
    \vdots   \\
    \vdots   \\
    {{{\bf{0}}_r}}  \\

 \end{matrix} } \right)= \left( {\begin{matrix}
    {\bf{w}}_{i}\\
   {{{\bf{v}}_{i}}}  \\
   {{{\bf{v}}_{i}}}  \\
    \vdots   \\
   {{{\bf{v}}_{i}}}  \\  \\

 \end{matrix} } \right) =   {\bf{w}}_{i}\oplus \mathop  \bigoplus \limits_{j = 0}^{k - 1} {{{\bf{v}}_{i}}}.$$
 Thus we have found $n$ (generalized) eigenvectors of the original matrix $M$.  In order to show this is a complete (generalized) eigenbasis, we need to show that \eqref{eq:vect} is a set of linearly independent vectors.  To do so let $E_0$ be the $(N+r)\times (N+r)$ matrix formed from the eigenbasis vectors of the divisor matrix $M_\phi$ and let $E_i$ for $1\leq i\leq k-1$ be the $r\times r$ matrices formed from the eigenbasis vectors of $B_i$, i.e $E_0=\left[\mathbf{u}_{0,1} \ \mathbf{u}_{0,2} \dots \mathbf{u}_{N+r}\right]\ ,\  E_i=
 \left[
 \mathbf{u}_{i,1} \ \mathbf{u}_{i,2}  \dots  \mathbf{u}_{i,r} \\
 \right]$.
 Let $E$ denote the matrix built from the vectors in Equation \eqref{eq:vect} as the columns. Thus we can write $E$ in the following block form
 $$E=\left[ \begin{matrix}
E_0 & 0  & 0 & 0 &\dots & 0\\
V & E_1  & E_2 &E_3 &\dots & E_{k-1}\\
V & \omega E_1  & \omega^2 E_2 & \omega^3 E_3 &\dots & \omega^{k-1} E_{k-1}\\
V & \omega^2 E_1  & \omega^4 E_2 & \omega^6 E_3 &\dots & \omega^{2(k-1)} E_{k-1}\\
\vdots & \vdots  & \vdots & \vdots & \ddots & \vdots\\
V & \omega^{k-1} E_1  & \omega^{2(k-1)} E_2 & \omega^{3(k-1)} E_3 &\dots & \omega^{(k-1)^2} E_{k-1}\\
 \end{matrix}\right]\s \s ,\s \s$$
where $V=\left[ \mathbf{v}_1 \ \  \mathbf{v}_2 \ \ \dots \  \ \mathbf{v}_{N+r} \right]$.  Showing that the vectors in \eqref{eq:vect} are linearly independent, is equivalent to showing that $\det{E}\neq 0$. Here, we notice that
\begin{align*}
\det{(E)}&=\det{\left( S
\left[\begin{matrix}
E_0 & 0 & 0 & \dots & 0\\
0 & E_1 & 0 & \dots & 0\\
0 & 0 & E_2 & \dots & 0\\
\vdots & \vdots & \vdots & \ddots & 0\\
0 & 0 & 0 & \dots & E_{k-1}\\
\end{matrix}\right]\right)}=\det(S)\det{\left[\begin{matrix}
E_0 & 0 & 0 & \dots & 0\\
0 & E_1 & 0 & \dots & 0\\
0 & 0 & E_2 & \dots & 0\\
\vdots & \vdots & \vdots & \ddots & 0\\
0 & 0 & 0 & \dots & E_{k-1}\\
\end{matrix}\right]}\\
&=\det(S)\prod_{j=0}^{k-1}\det(E_j)
\end{align*}
We have shown previously that the columns of $S$ are orthogonal, thus $\det(S)\neq 0$.  Also we chose the columns of $E_i$, for $0\leq i \leq k-1$, to be generalized eigenbases. Thus we can guarantee that, for every $i$, $\det{E_i}\neq 0$ and thus $\det{E}\neq0$.  This proves that that the set of $n$ vectors we have found actual constitutes a (generalized) eigenbasis for the matrix $M$.

 One can check that the (generalized) eigenvectors correspond to the eigenvalues, as stated in the theorem, by showing that $$(M-I\lambda_{m,l})^{t_{m,l}}\left[{\bf{0}}_N \, \oplus \left[ {\mathop  \bigoplus \limits_{j = 0}^{k - 1} {\omega ^{mj}}{{\bf{u}}_{m,l}}} \right] \right]=\mathbf{0} \ \ \text{for} \ \ t_{m,l}=\text{(rank of } \textbf{x}_{m,l}\text{)}$$
and also that
$$(M-I\lambda_{0,i})^{t_{0,i}}\left[{{\bf{w}}_i}\,\oplus \left[   {\mathop  \bigoplus \limits_{j = 0}^{k - 1} {{\bf{v}}_i}} \right] \right]=\mathbf{0} \ \ \text{for} \ \ t_{0,i}=\text{(rank of } \textbf{x}_{0,i}\text{)}.$$
\end{proof}

\begin{example}\label{ex:eigvec}
As an illustration of Theorem \ref{eigvec} we again consider the graph $G$ shown in Figure \ref{fig:ex1}. In Example \ref{ex:sequence} we found the  decomposition $A(1)=A(0)_{\psi}\oplus B(0)_1\oplus B(0)_2$ of the adjacency matrix $A=A(G)$ over the basic automorphism $\psi=\phi^2=(2,8,5)(3,9,7)(4,10,6)$  (see Equation \eqref{eq:7}). Eigenbases corresponding to $B_{\psi}$, $B_1$, and $B_2$, respectively, are given by the vectors $\mathbf{u}_{0,i}$, $\mathbf{u}_{1,i}$, and $\mathbf{u}_{2,i}$  where
$$B_\psi=\left[ \begin{matrix} 0 & 3 & 0 & 0  \\
   1 & 2 & 1 & 1  \\
   0 & 1 & 0 & 0  \\
   0 & 1 & 0 & 0  \end{matrix}\right]\qquad
   \begin{array}{lll}
   \textbf{u}_{0,1}=(3,\ 1+\sqrt{6},\ 1,\ 1)^{T} \ & \textbf{w}_{1}=(3)  & \textbf{v}_1=(1+\sqrt{6},\ 1,\ 1)^{T}\\
   \textbf{u}_{0,2}=(3,\ 1-\sqrt{6},\ 1,\ 1)^{T} \ & \textbf{w}_{2}=(3) & \textbf{v}_2=(1-\sqrt{6},\ 1,\ 1)^{T}\\
   \textbf{u}_{0,3}=(-1,\ 0,\ 0,\ 1)^{T} \ &\textbf{w}_{3}=(-1) & \textbf{v}_3=(0,\ 0,\ 1)^{T}\\
   \textbf{u}_{0,4}=(-1,\ 0,\ 1,\ 0)^{T} \ &\textbf{w}_{4}=(-1) & \textbf{v}_4=(0,\ 1,\ 0)^{T}\\
   \end{array}
   $$
   $$B_1=B_2=\left[ \begin{matrix} { - 1} & 1 & 1   \\
   1 & 0 & 0  \\
    1 & 0 & 0 \end{matrix}\right] \qquad
   \begin{array}{lll}
   \textbf{u}_{1,1}=\textbf{u}_{2,1}=(-2,\ 1,\  1)^{T}\\
   \textbf{u}_{1,2}=\textbf{u}_{2,2}=(1,\ 1,\  1)^{T}\\
   \textbf{u}_{1,3}=\textbf{u}_{2,3}=(0,\ -1,\  1)^{T}\\
   \end{array}.$$
Note that $\textbf{w}_i$ has only one component since the first basic automorphism only fixed one vertex, i.e. $N=1$. Using the formula in Theorem \ref{eigvec} an eigenbasis of the original matrix $A$ is given by
\begin{align*}
&{{\bf{w}}_1}\oplus {{\bf{v}}_1}\oplus {{\bf{v}}_1}\oplus {{\bf{v}}_1}= (3,\ 1+\sqrt{6},\ 1,\ 1,\ 1+\sqrt{6},\ 1,\ 1,\ 1+\sqrt{6},\ 1,\ 1)^{T}\\
&{{\bf{w}}_2}\oplus {{\bf{v}}_2}\oplus {{\bf{v}}_2}\oplus {{\bf{v}}_2}=(3,\ 1-\sqrt{6},\ 1,\ 1,\ 1-\sqrt{6},\ 1,\ 1,\ 1-\sqrt{6},\ 1,\ 1)^{T}\\
&{{\bf{w}}_3}\oplus {{\bf{v}}_3}\oplus {{\bf{v}}_3}\oplus {{\bf{v}}_3}= (-1,\ 0,\ 0,\ 1,\ 0,\ 0,\ 1,\ 0,\ 0,\ 1)^{T}\\
&{{\bf{w}}_4}\oplus {{\bf{v}}_4}\oplus {{\bf{v}}_4}\oplus {{\bf{v}}_4}= (-1,\ 0,\ 1,\ 0,\ 0,\ 1,\  0,\ 0,\ 1,\  0)^{T}\\
&{{\bf{0}}_N}\oplus\textbf{u}_{1,1}\oplus\omega\textbf{u}_{1,1}\oplus\omega^2\textbf{u}_{1,1}=(0,\ -2,\ 1,\  1,\ -2\omega,\ \omega,\  \omega,\ -2\omega^2,\ \omega^2,\ \omega^2)^{T}\\
&{{\bf{0}}_N}\oplus\textbf{u}_{1,2}\oplus\omega\textbf{u}_{1,2}\oplus\omega^2\textbf{u}_{1,2}=(0,\ 1,\ 1,\ 1,\omega,\ \omega,\  \omega,\ \omega^2,\ \omega^2,\  \omega^2)^{T}\\
&{{\bf{0}}_N}\oplus\textbf{u}_{1,3}\oplus\omega\textbf{u}_{1,3}\oplus\omega^2\textbf{u}_{1,3}=(0,\ 0,\ -1,\  1,\ 0,\ -\omega,\  \omega,\ 0,\ -\omega^2,\  \omega^2)^{T}\\
&{{\bf{0}}_N}\oplus\textbf{u}_{2,1}\oplus\omega^2\textbf{u}_{2,1}\oplus\omega^4\textbf{u}_{2,1}=(0,\ -2,\ 1,\  1,\ -2\omega^2,\ \omega^2,\  \omega^2,\ -2\omega,\ \omega,\  \omega)^{T}\\
&{{\bf{0}}_N}\oplus\textbf{u}_{2,2}\oplus\omega^2\textbf{u}_{2,2}\oplus\omega^4\textbf{u}_{2,2}=(0,\ 1,\ 1,\  1,\omega^2,\ \omega^2,\  \omega^2,\ \omega,\ \omega,\  \omega)^{T}\\
&{{\bf{0}}_N}\oplus\textbf{u}_{2,3}\oplus\omega^2\textbf{u}_{2,3}\oplus\omega^4\textbf{u}_{2,3}=(0,\ 0,\ -1,\  1,\ 0,\ -\omega^2,\  \omega^2,\ 0,\ -\omega,\  \omega)^{T}\\
\end{align*}
where $\omega=e^{2\pi i/3}$.
\end{example}

The process  carried out in Example \ref{ex:eigvec} of constructing eigenvectors of a matrix from the eigenvectors of its decomposition {over a basic automorphism} can also be done for separable automorphisms. This is done by finding the eigenvectors of the matrices in the final decomposition and working backwards, as in this example, until the eigenvectors of the original matrix have been fully reconstructed. For instance, we could find the eigenvectors corresponding to the full decomposition of the matrix $A$ over the automorphism $\phi$ in Example \ref{ex:sequence}. From there we could work back to an eigenbasis of $A$.

If $\phi$ is a uniform automorphism of a graph $G$ then the eigenvectors of an automorphism compatible matrix $M$ can also be decomposed as is shown in \eqref{eq:vect} for $N=0$. Specifically, if $M_{\phi}\oplus B_1 \oplus B_2 \oplus \cdots \oplus B_{k-1}$ is an equitable decomposition of $M$ with respect to $\phi$ in which $\mathbf{u}_{m,\ell}$ is the $\ell^{th}$ eigenvector of $B_m$ then the eigenvectors of $M$ are given by the set
$$\left\{\ {{\mathop  \bigoplus \limits_{j = 0}^{k - 1} {\omega ^{m j}}{{\bf{u}}_{m,\ell}}}:1 \leq m \leq k - 1,1 \leq \ell \leq r, \omega=e^{2\pi i/k} \,} \right\}.$$

It is worth noting that both the eigenvalues and eigenvectors of $M$ are \emph{global} characteristics of the matrix $M$ in the sense that they depend, in general, on all entries of the matrix or equivalently on the entire structure of the graph $G$. In contrast, many symmetries of a graph $G$ are inherently \emph{local}, specifically when they correspond to an automorphism that fixes some subset of the vertex set of $G$, e.g. a basic automorphism.

This difference is particularly important in the case where we are interested in deducing spectral properties associated with the graph structure of a network. Reasons for this include the fact that most real networks are quite large, often having either thousands, hundreds of thousands, or more vertices. Second, real networks are on average much more structured and in particular have more symmetries than random graphs (see \cite{MacArthur}). Third, there is often only partial or local information regarding the structure of many of these networks because of the complications in obtaining network data (see, for instance, \cite{Clauset08}).

The implication, with respect to equitable decompositions, is that by finding a local graph symmetry it is possible to gain information regarding the graph's set of eigenvalues and eigenvectors, which is information typically obtained by analyzing the entire structure of the network. This information, although incomplete, can be used to determine key spectral properties of the network.

One of the most important and useful of these characteristics is the spectral radius associated with the graph structure $G$ of a network. The \emph{spectral radius} of a matrix $M$ associated with $G$ is given by
\[\rho(M)=\max\{|\lambda|:\lambda\in\sigma(M)\}.\]

The spectral radius $\rho(M)$ of a network, {or, more generally, a dynamical system}, is particularly important for studying the system's dynamics. For instance, the matrix $M$ associated with a network may be a global or local linearization of the system of equations that govern the network's dynamics. If the network's dynamics are modeled by a discrete-time system, then stability of the system is guaranteed if $\rho(M)<1$ and local instability results when $\rho(M)>1$ \cite{BW10,BW11,HA03}.

Using the theory of equitable decompositions, it is possible to show not only that $\sigma(M_{\phi})\subset\sigma(M)$, but also that the spectral radius $\rho(M)$ of $M$ is an eigenvalue of $M_{\phi}$ if $M$ is both nonnegative and irreducible.

\begin{prop}\label{lem:3}\textbf{(Spectral Radius of Equitable Partitions)}
Let $\phi$ be a basic or separable automorphism of a graph $G$ with $M$ an automorphism compatible matrix. If $M$ is nonnegative, then $\rho(M)=\rho(M_{\phi})$. If $M$ is both irreducible and nonnegative then the spectral radius, $\rho(M)$, is an eigenvalue of $M_{\phi}$.
\end{prop}

\begin{proof}
We begin by proving the result for basic (and uniform) automorphisms and then extending the result to separable automorphisms. Assume $M$ is nonnegative and with basic automorphism $\phi$. To prove that $\rho(M)=\rho(M_{\phi})$, we first claim that $\rho(M_{\phi})\geq\rho(B_j)$ for $1\leq j\leq k-1$ where $M_{\phi}\oplus B_1\oplus\cdots\oplus B_{k-1}$ is an equitable decomposition of $M$ with respect to $\phi$.

To verify this claim, we first need Corollary 8.1.20 in \cite{Horn85} which states that if $N$ is a principal submatrix of $M$ then $\rho(N)\leq\rho(M)$ if $M$ is nonnegative. Recall that $M_\phi=\left[ \begin{matrix}
F & kH\\
L & B_0\\
\end{matrix} \right]$ if $\phi$ is a basic automorphism that fixes some positive number of vertices.
Thus, for a basic automorphism, $B_0$ is a principal submatrix of $M_\phi$. Since $M$ is nonnegative,  Equation \eqref{eq:B} shows that $M_{\phi}$ is nonnegative, and  we can conclude that $\rho(B_0)\leq\rho(M_\phi)$.  In the case that $\phi$ is a uniform automorphism, $M_{\phi}=B_0 $.
 
Next, for a matrix $P\in\mathbb{C}^{n\times n}$, let $|P|\in\mathbb{R}_{\geq 0}^{n\times n}$ denote the matrix with entries $|P|_{ij}=|P_{ij}|$, i.e. $|P|$ is the entrywise absolute value of $P$. Moreover, if $P,Q\in\mathbb{R}^{n\times n}$ let $P\leq Q$ if $P_{ij}\leq Q_{ij}$ for all $1\leq i,j\leq n$. Theorem 8.1.18 in \cite{Horn85}, states that if $|P|\leq Q$  then $\rho(P)\leq\rho(Q)$. Because
$$\left| {{B_j}} \right| = \left| {\sum\limits_{m = 0}^{k - 1} {{{\left( {{\omega ^j}} \right)}^m}{M_m}} } \right| \leq \sum\limits_{m = 0}^{k - 1} {\left| {{{\left( {{\omega ^j}} \right)}^m}{M_m}} \right|}  = \sum\limits_{m = 0}^{k - 1} {{M_m}}  = {B_0}$$
we can conclude that $\rho(B_j)\leq\rho(B_0)$ for all $1\leq j\leq k-1$. Therefore,
\begin{equation}\label{eq:rho}
\rho(B_j)\leq\rho(B_0)\leq\rho(M_{\phi}) \ \text{for all} \ 1\leq j\leq k-1.
\end{equation} 
which verifies our claim.  Using this claim and the fact that $\sigma(M)=\sigma(M_\phi)\cup\sigma(B_1)\cup\dots\cup\sigma(B_{k-1})$ we can immediately conclude that $\rho(M)=\rho(M_\phi)$.

Now we assume that $M$ is both nonnegative and irreducible. The Perron-Frobenius Theorem implies that $r=\rho(M)$ is a simple eigenvalue of $M$. 

Next we claim that if $M$ is irreducible, then $M_\phi$ is also irreducible. Let $M$ be an $n\times n$ nonnegative matrix with an associated basic automorphism $\phi$ with order $k$, and let $M_\ell$ be the submatrix of $M$ associated with the $\ell^{th}$ power of the semitransversal in an equitable decomposition of $M$ over $\phi$. Recall that a matrix is reducible if and only if its associated weighted digraph is \emph{strongly connected}, meaning for any two vertices in the graph there is a directed path between them.  Suppose $G$ is the strongly connected graph with weighted adjacency matrix $M$.  Also we suppose that $G_\phi$ is the graph whose weighted adjacency matrix is $M_\phi$ and let $a$ and $b$ be vertices of $G_\phi$. Note that every vertex fixed by $\phi$ in $G$ directly corresponds to a vertex in $G_\phi$, and all other vertices correspond to a collection of $k$ vertices in $G$, (cf. Section \ref{sec:6}).  Choose $a'$ and $b'$ in $G$ to be any vertices corresponding to $a$ and $b$, respectively. Now because $G$ is strongly connected it contains a path $a'=v_0', v_1', v_2',\dots, b'=v_m'$ from $a'$ to $b'$.  Consider the sequence of vertices $a=v_0, v_1, v_2 ,\dots, b=v_m$ where $v_i$ is the unique vertex in $G_\phi$ corresponding $v_i'$ in $G$. To prove this is a path we must show each of the entries in matrix $M_\phi$ corresponding to the edges $v_i\rightarrow v_{i+1}$ are positive. If $v_i'$ or $v_{i+1}'$ are fixed by $\phi$ then the entry in $M$ corresponding to the $v_i'\rightarrow v_{i+1}'$ edge is either equal to the entry in $M_\phi$ corresponding to the $v_i\rightarrow v_{i+1}$ edge, or is a positive multiple thereof. If $v_i'$ and $v_{i+1}'$ are not fixed by $\phi$, then suppose the $v_i'\rightarrow v_{i+1}'$ edge corresponds to $M_\ell(r,s)$ for some $\ell$ and for some indices $r$ and $s$. By hypothesis, $M_\ell(r,s)>0$. The $v_i\rightarrow v_{i+1}$ edge corresponds to $B_0(r,s)$, and  $B_0=\sum_\ell{M_\ell}$, {where each $M_\ell(r,s)$ is nonnegative}. Therefore this entry must also be positive in $M_\phi$.  Thus we can conclude that $G_\phi$ is strongly connected and therefore $M_\phi$ is irreducible.


Because $M_\phi$ is irreducible, we can apply the Perron-Frobenius Theorem to $M_\phi$. This implies that $\rho(M_\phi)$ is an eigenvalue of $M_\phi$, but from the first part of this theorem, we already showed that $\rho(M)=\rho(M_\phi)$, thus we conclude that $\rho(M)$ must be an eigenvalue of $M_\phi$.

This completes the proof if $\phi$ is a basic automorphism. If $\phi$ is separable then Proposition \ref{thm:3} guarantees that there are basic automorphisms $\psi_0,\dots,\psi_{h}$ that induce a sequence of equitable decompositions on $M$ such that $M_{\phi}=(\dots(M_{\psi_0})_{\psi_1}\dots)_{\psi_{h}}$. By induction each subsequent decomposition results in a nonnegative divisor matrix $(\dots(M_{\psi_0})_{\psi_1}\dots)_{\psi_i}$ for $i\leq h$ with the same spectral radius $r=\rho(M)$ implying that $\rho(M_{\phi})=\rho(M)$ for any $\phi\in Aut(G)$.
\end{proof}

It is worth noting that many matrices typically associated with real networks are both nonnegative and irreducible. This includes the adjacency matrix as well as other weighted matrices \cite{Newman10}; although, there are some notable exceptions, including Laplacian matrices. Moreover, when analyzing the stability of a network, a linearization $M$ of the network's dynamics inherits the symmetries of the network's structure. Hence, if a symmetry of the network's structure is known then this symmetry can be used to decompose $M$ into a smaller divisor matrix $M_{\phi}$. As $M$ and $M_{\phi}$ have the same spectral radius, under the conditions stated in Proposition \ref{lem:3}, then one can use the smaller matrix $M_{\phi}$ to either calculate or estimate the spectral radius of the original unreduced network as is demonstrated in the following example.

\begin{example}\label{ex:sr}
We again consider the graph $G$ from Example \ref{ex:sequence} with $\phi=(2,5,8)(3,4,6,7,9,10)$. {Here the graph's adjacency matrix $A=A(G)$ is both irreducible and nonnegative. The set of eigenvalues of the matrix $A$} and the divisor matrix $A_{\phi}$ are
\[
\sigma(A)=\{ 1 + \sqrt{6},\ 1 -\sqrt{6},\ -2,\ -2,\ 1,\ 1,\ 0,\ 0,\ 0,\ 0\} \ \ \text{and} \ \ \sigma(A_{\phi})=\{ 1 + \sqrt{6},\ 1 -\sqrt{6},\ 0,\ 0\}.
\]
Hence, $\rho(A)=\rho(A_{\phi})$ and $\rho(A_\phi)$ is actually an eigenvalue of $A$ as guaranteed by Proposition \ref{lem:3}.
\end{example}

\section{Equitable Decompositions and Improved Eigenvalues Estimates}\label{sec:5}
Beginning in the mid-19th century a number of methods were developed to approximate the eigenvalues of general complex valued matrices. These included the results of Gershgorin \cite{Gershgorin31}, Brauer \cite{Brauer47}, Brualdi \cite{Brualdi82}, and Varga \cite{Varga99}. The main idea in each of these methods is that for a matrix $M\in\mathbb{C}^{n\times n}$ it is possible to construct a bounded region in the complex plane that contains the matrix' eigenvalues. This region serves as an approximation for the eigenvalues of $M$.

In this section we investigate how equitable decompositions affect, in particular, the approximation method of Gershgorin. Our main result is that the Gershgorin region associated with an equitable decomposition of a matrix $M$ is contained in the Gershgorin region associated with $M$. That is, by equitably decomposing a matrix over some automorphism it is possible to gain improved eigenvalue estimates by use of Gershgorin's theorem.

To describe this result we first give the following classical result of Gershgorin.

\begin{thm}{\textbf{(Gershgorin's Theorem)} \cite{Gershgorin31}} Let $M\in \mathbb{C}^{n\times n}$. Then all eigenvalues of $M$ are contained in the set
$$\Gamma (M)= \bigcup_{i=1}^n \left\{ \lambda \in \mathbb{C} :| \lambda-M_{ii}|\leq \sum_{j=1,j\ne i}^n | M_{ij} | \right\}$$
\end{thm}

Geometrically this theorem states that all eigenvalues of a given matrix $M\in\mathbb{C}^{n\times n}$ must lie in the union of $n$ disks in the complex plane, where the $i^{th}$ disk is constructed from the $i^{th}$ row of $M$. Specifically, the $i^{th}$ disk is centered at $M_{ii}\in\mathbb{C}$ and has the radius $\sum_{j=1,j\ne i}^n |M_{ij}|$. The union of these disks forms the \emph{Gershgorin region} $\Gamma(M)$ of the matrix $M$. The following theorem describes the effect that an equitable decomposition has on a matrix' Gershgorin region.

\begin{thm}\label{thm:Gersh}
Let $\phi$ be a basic or separable automorphism of a graph $G$ with M an automorphism compatible matrix. If $B=M_{\phi}\oplus B_1\oplus \cdots \oplus B_{k}$ is an equitable decomposition of $M$ with respect to $\phi$ then $\Gamma(B)\subseteq\Gamma(M)$.
\end{thm}

\begin{proof}
First, suppose for simplicity that $\phi$ is a uniform automorphism of $G$. The $i^{th}$ row of the matrix $M$ defines a disk in the complex plane centered at $M_{ii}$ with radius equal to $\sum_{j\ne i}\left| M_{ij}\right| $.  So we want every disk generated by each $B_t$ matrix to be contained in some disk generated by $M$. We can achieve this if the distance between disk centers is less than the difference in the two disk's radii. Thus we need to show that for every $i$ and $t$ that $\left| {{M_{qq}} - {{\left[ {{B_t}} \right]}_{ii}}} \right| \leq \left| \sum\limits_{j \ne q} {\left| {{M_{qj}}} \right|}  - \sum\limits_{j \ne i} {\left| {{{\left[ {{B_t}} \right]}_{ij}}} \right|} \right| $ for some $q$. Let $q=i$. Using Equation \eqref{eq:B} and rearranging terms we see that
\begin{align}
 \left| \sum\limits_{j \ne q} {\left| {{M_{qj}}} \right|}  - \sum\limits_{j \ne i} {\left| {{{\left[ {{B_t}} \right]}_{ij}}} \right|} \right| \geq &
  \sum\limits_{l \ne i}^n {\left| {{M_{il}}} \right|}  - \sum\limits_{j \ne i}^r {\left| {{{\left[ {{B_t}} \right]}_{ij}}} \right|}\cr
  =& \left[ {\sum\limits_{m = 0}^{k - 1} {\sum\limits_{j = 1}^r {\left| {{{\left[ {{M_m}} \right]}_{ij}}} \right|} }  - \sum\limits_{j \ne i}^r {\left| {\sum\limits_{m = 0}^{k - 1} {{{\left( {{\omega ^t}} \right)}^m}{{\left[ {{M_m}} \right]}_{ij}}} } \right|} } \right] - \left| {{{\left[ {{M_0}} \right]}_{ii}}} \right| \cr
   =& \left[ {\sum\limits_{m = 0}^{k - 1} {\sum\limits_{j = 1}^r {\left| {{{\left( {{\omega ^t}} \right)}^m}{{\left[ {{M_m}} \right]}_{ij}}} \right|} }  - \sum\limits_{j \ne i}^r {\left| {\sum\limits_{m = 0}^{k - 1} {{{\left( {{\omega ^t}} \right)}^m}{{\left[ {{M_m}} \right]}_{ij}}} } \right|} } \right] - \left| {{{\left[ {{M_0}} \right]}_{ii}}} \right| \cr
   \geq & \left[ {\sum\limits_{m = 0}^{k - 1} {\sum\limits_{j = 1}^r {\left| {{{\left( {{\omega ^t}} \right)}^m}{{\left[ {{M_m}} \right]}_{ij}}} \right|} }  - \sum\limits_{j \ne i}^r {\sum\limits_{m = 0}^{k - 1} {\left| {{{\left( {{\omega ^t}} \right)}^m}{{\left[ {{M_m}} \right]}_{ij}}} \right|} } } \right] - \left| {{{\left[ {{M_0}} \right]}_{ii}}} \right| \cr
   =& \sum\limits_{m = 0}^{k - 1} {\left[ {\sum\limits_{j = 1}^r {\left| {{{\left( {{\omega ^t}} \right)}^m}{{\left[ {{M_m}} \right]}_{ij}}} \right|}  - \sum\limits_{j \ne i}^r {\left| {{{\left( {{\omega ^t}} \right)}^m}{{\left[ {{M_m}} \right]}_{ij}}} \right|} } \right]}  - \left| {{{\left[ {{M_0}} \right]}_{ii}}} \right| \cr
   =& \sum\limits_{m = 0}^{k - 1} {\left[ {\left| {{{\left( {{\omega ^t}} \right)}^m}{{\left[ {{M_m}} \right]}_{ii}}} \right|} \right]}  - \left| {{{\left[ {{M_0}} \right]}_{ii}}} \right| \cr
   =& \sum\limits_{m = 1}^{k - 1} {\left| {{{\left( {{\omega ^t}} \right)}^m}{{\left[ {{M_m}} \right]}_{ii}}} \right|}  \cr
   \geq & \left| {\sum\limits_{m = 1}^{k - 1} {{{\left( {{\omega ^t}} \right)}^m}{{\left[ {{M_m}} \right]}_{ii}}} } \right| \cr
   =& \left| {{{\left[ {{B_t}} \right]}_{ii}} - {M_{ii}}} \right| \cr
   =& \left| {{M_{ii}} - {{\left[ {{B_t}} \right]}_{ii}}} \right| \cr
\end{align}
where the $M_m$ matrices are the block matrices found in the block circulant form of $M$. Therefore, every disk generated by $B_m$ for each $m$ is contained in some disk generated by $M$ and we conclude that $\Gamma(B)\subset \Gamma(M)$ for a uniform automorphism.\\

Next we assume our graph has a basic automorphism $\phi$ that fixes a positive number of vertices of $G$. Then the matrix $M$ decomposes in the following way
\begin{equation*}
 \left[\begin{array}{llllll}
F & H & H & H & \cdots & H \\
L & M_0 & M_1 & M_2 & \cdots & M_{k-1} \\
L & M_{k-1} & M_0 & M_1 & \cdots & M_{k-2} \\
L & M_{k-2} & M_{k-1} & M_0 & \cdots & M_{k-3} \\
\vdots & \vdots & \vdots & \vdots & & \vdots \\
L & M_1 & M_2 & M_3 & \cdots & M_0 \\
\end{array}\right]\rightarrow
\left[\begin{array}{llllll}
F & kH & 0 & 0 & \cdots & 0 \\
L & B_0 & 0 & 0 & \cdots & 0 \\
0 & 0 & B_1 & 0 & \cdots & 0 \\
0 & 0 & 0 & B_2 & \cdots & 0 \\
\vdots & \vdots & \vdots & \vdots & & \vdots \\
0 & 0 & 0 & 0 & \cdots & B_{k-1} \\
\end{array}\right]
\end{equation*}
Clearly the first $N$ rows, which correspond to vertices fixed by the automorphism, have exactly the same Gershgorin region in both matrices. {For the next block row, the previous argument implies that the Gershgorin disks associated with these rows are contained in the Gershgorin disks of the corresponding rows in $M$ if we disregard the first block column. Including the first block column in the Gershgorin region calculation increases the radii of the corresponding Gershgorin disks for both matrices by the same amount.} Thus we still get containment. For all remaining rows the same argument applies except we only increase the radii of the disks from the original matrix when considering the first block column. Thus, for the basic automorphism $\phi$ we have $\Gamma(B) \subset \Gamma( M )$.

We showed in Proposition \ref{thm:3} that decomposing a matrix over a separable automorphism can be done as a sequence of decompositions of basic automorphisms. After each decomposition of a basic automorphism we get containment. Thus, using an inductive argument we can extend this theorem to all separable automorphisms.
\end{proof}

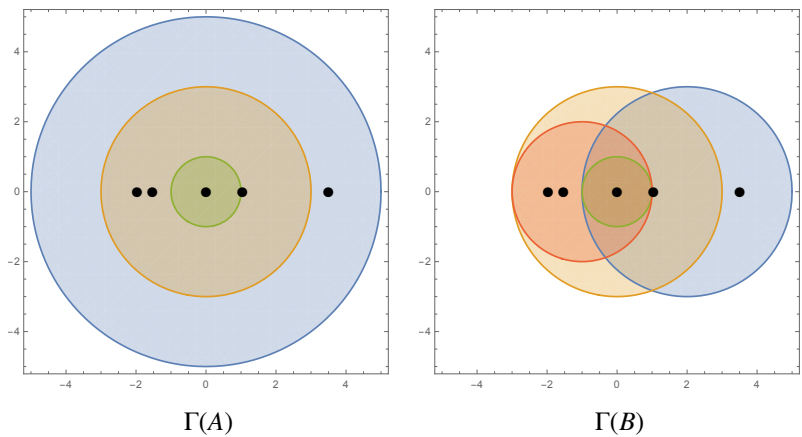
\begin{figure}
\begin{center}
\begin{tabular}{cc}
    \begin{overpic}[scale=.4]{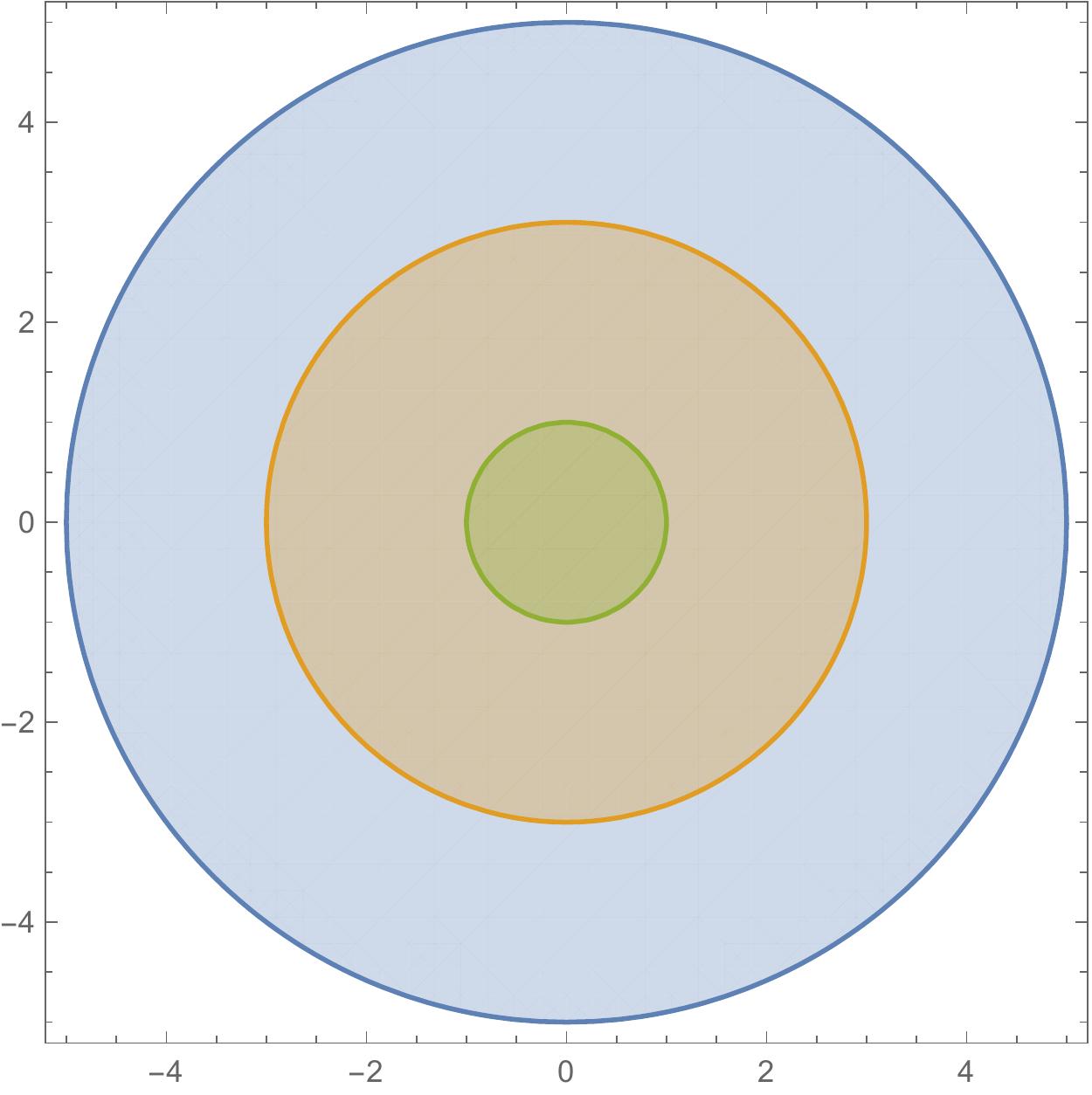}
    \put(46,-10){$\Gamma(A)$}
    \put(32,50.25){$\bullet$}
    \put(36,50.25){$\bullet$}
    \put(50,50.25){$\bullet$}
    \put(59.5,50.25){$\bullet$}
    \put(82,50.25){$\bullet$}
    \end{overpic} &
    \begin{overpic}[scale=.4]{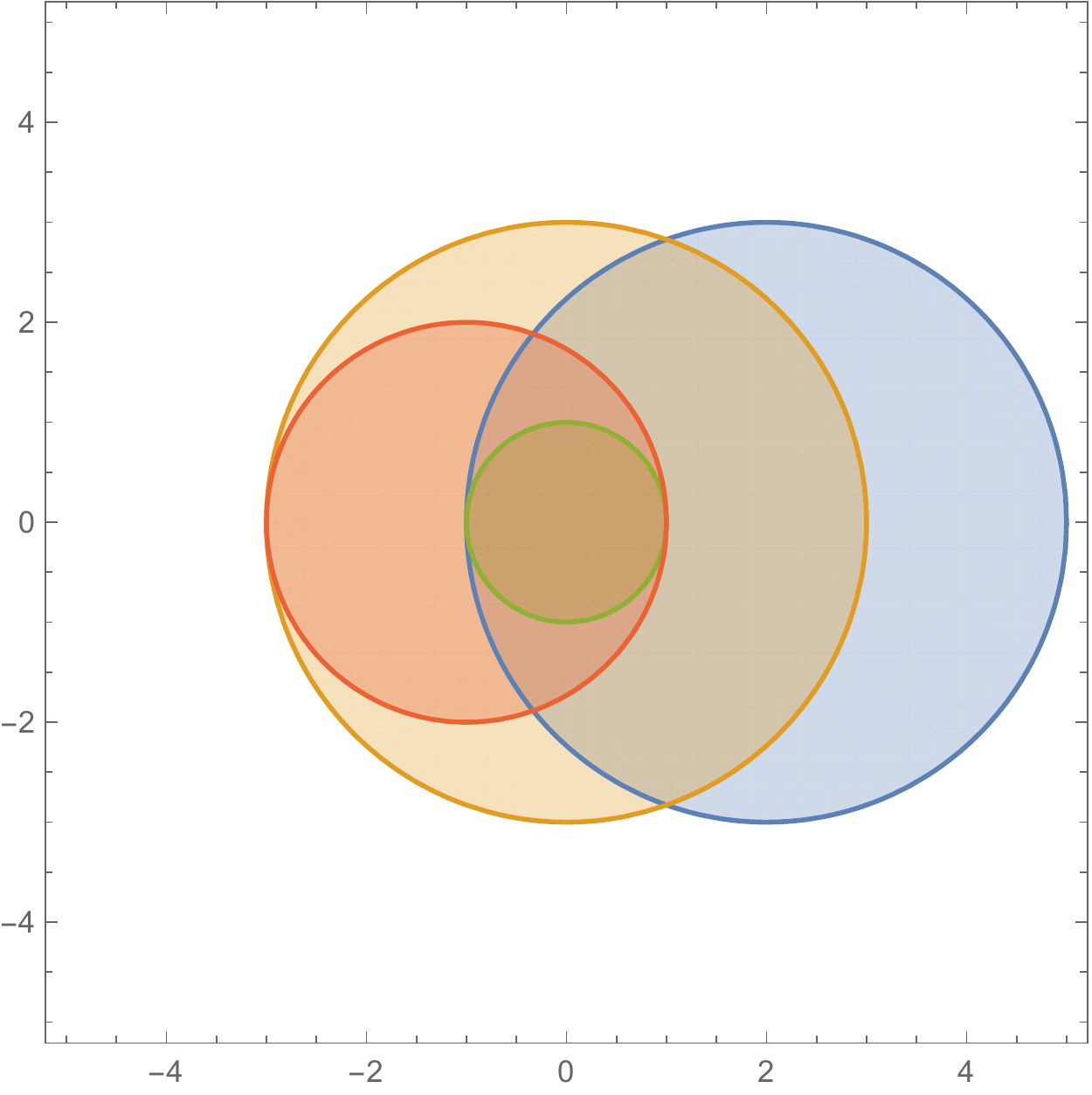}
    \put(46,-10){$\Gamma(B)$}
    \put(32,50.25){$\bullet$}
    \put(36,50.25){$\bullet$}
    \put(50,50.25){$\bullet$}
    \put(59.5,50.25){$\bullet$}
    \put(82,50.25){$\bullet$}
    \end{overpic}
\end{tabular}
\end{center}
\caption{The Gershgorin regions $\Gamma(A)$ and $\Gamma(B)$ each made up of a union of disks corresponding to the adjacency matrix $A=A(G)$ of the graph $G$ in Figure \ref{fig:ex1} and its equitable decomposition $B$ over the automorphism $\psi_0=(2,8,5)(3,9,7)(4,10,6)$, respectively. Black points indicate the eigenvalues $\sigma(A)=\sigma(B)$.}\label{fig:original}
\end{figure}

\begin{example}
Again we consider the graph $G$ shown in Figure \ref{fig:ex1} with basic automorphism $\psi_0=\phi^2=(2,8,5)(3,9,7)(4,10,6)$. The equitable decomposition of the adjacency matrix $A=A(G)$ with respect to $\psi_0$ is given by $B=A_{\psi_0}\oplus B_1\oplus B_2$ where $A_{\psi_0},B_1,B_2$ are given in Equation \eqref{eq:7}. The Gershgorin regions of both $A$ and $B$ are shown in Figure \ref{fig:original} where $\Gamma(B)\subset\Gamma(A)$. That is, the equitable decomposition of $A$ over $\psi_0$ results in an improved Gershgorin estimate of its eigenvalues.
\end{example}

The effectiveness of Gershgorin's theorem in estimating  a matrix's eigenvalues depends heavily on whether  the row sums $\sum_{j=1,j\ne i}^n |M_{ij}|$ are large or small. For example, if the matrix $M$ is the adjacency matrix of a graph $G$ then the $i^{th}$ \emph{disk} of $\Gamma(M)$ has a radius equal to the number of neighbors the vertex $v_i$ has in $G$. Real networks often contain a few vertices that have an abnormally high number of neighbors. These vertices, which are referred to as \emph{hubs}, greatly reduce the effectiveness of Gershgorin's theorem.

{One application of the theory of equitable decompositions is that it can be used to potentially reduce the size of the Gershgorin region associated with a graph's adjacency matrix $A=A(G)$. Note that this region is made up of disks where the radius of the $i^{th}$ disk is equal to the degree of the $i^{th}$ vertex. Hence, the graph's hubs generate the graph's largest Gershgorin disks. The strategy we propose is to find an automorphism $\phi$ of a graph $G$ that permutes the vertices adjacent to the graph's largest hub. After decomposing the adjacency matrix over $\phi$, the resulting decomposed matrix will potentially have a smaller row sum associated with this largest hub and consequently a smaller Gershgorin region. This process can be continued by finding automorphisms of the decomposed graph to further decompose and potentially further reduce the Gershgorin region associated with the decomposed matrix.}

This is demonstrated in the following example.

\begin{figure}
\begin{center}
\begin{tabular}{cc}
\begin{overpic}[scale=.11]{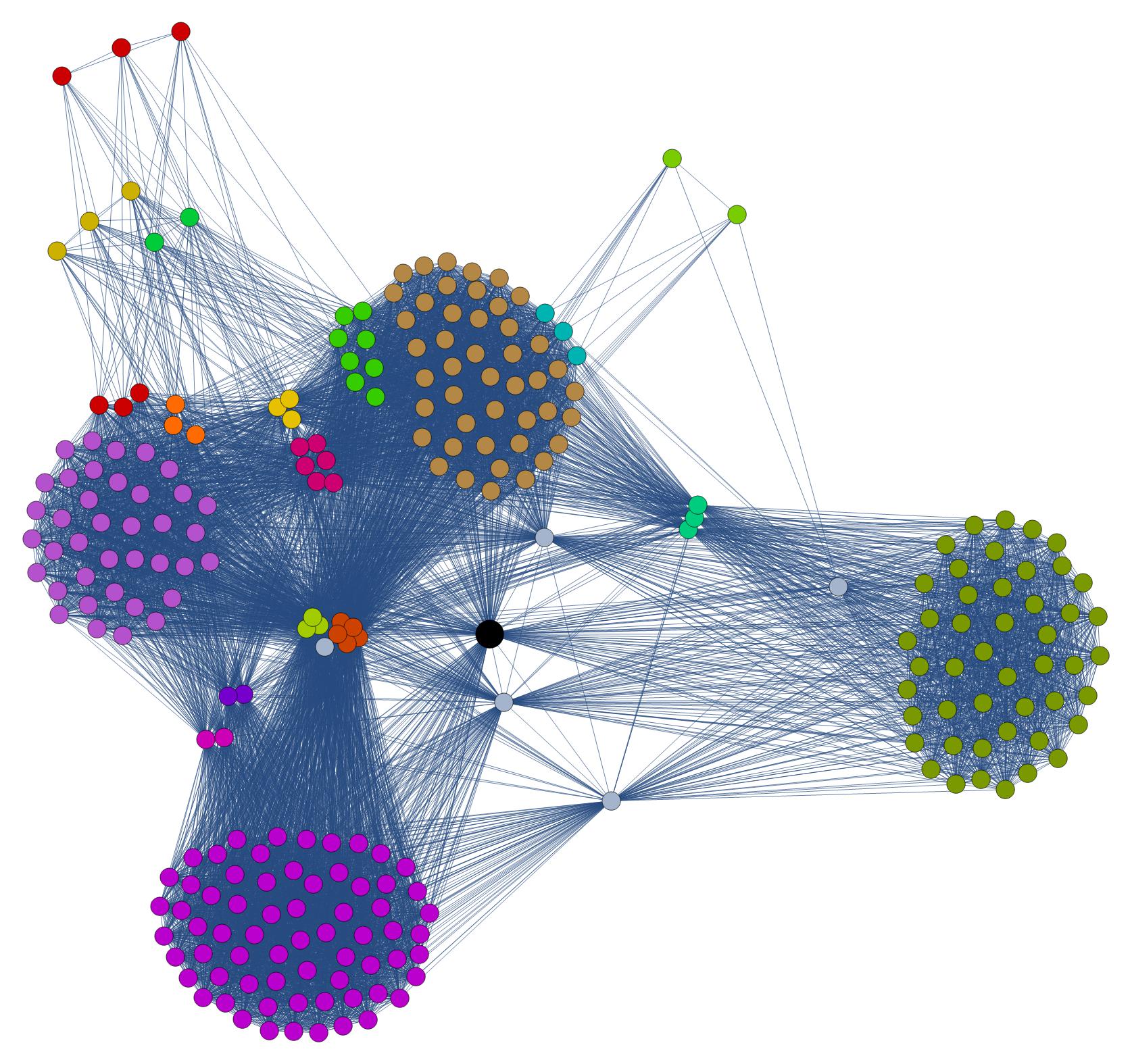}
\put(46,40){\tiny$v_{PR}$}
\end{overpic} & \vspace{0.2in}
\begin{overpic}[scale=.4]{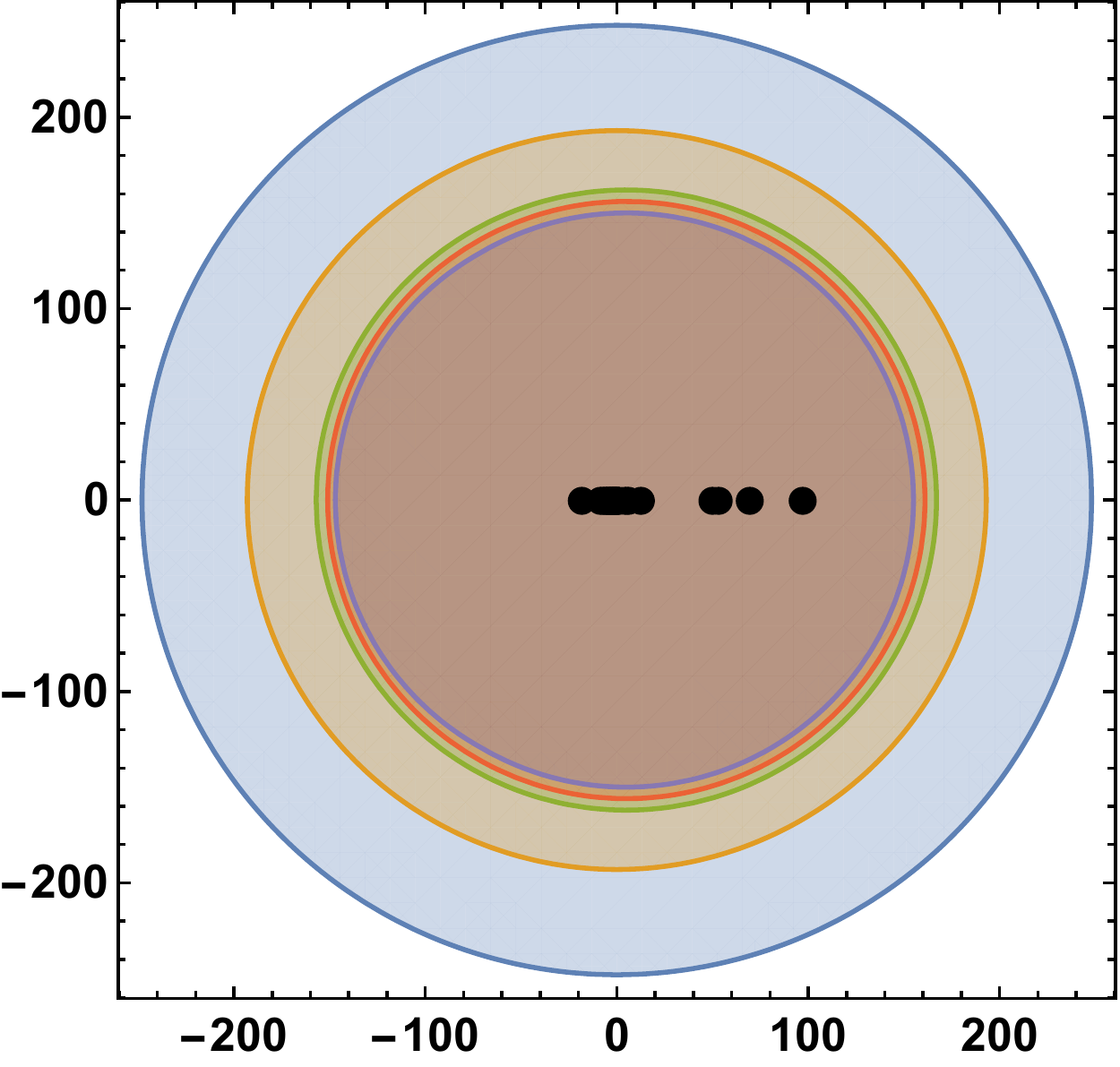}
\end{overpic}
\end{tabular}
\caption{Left: A graph representing 254 individuals belonging to seven different organizations in the Boston area prior to the American revolutionary war. Edges represent whether two individuals belonged to the same organization. The black vertex $v_{PR}$ represents Paul Revere. Right: The transposed Gersgorin regions corresponding to a sequence of equitable decompositions performed on the network's adjacency matrix, in which each subsequent decomposition results in a smaller region contained in the previous. Black points indicate the eigenvalues $\sigma(A)$}\label{fig:revwar}
\end{center}
\end{figure}

\begin{example}\label{ex:revwar}
Figure \ref{fig:revwar} (left) shows a social network of 254 members of seven different organizations in the Boston area prior to the American revolutionary war \cite{H13}. Each vertex represents a person and an edge between two people appears when these two people belong to the same organization. The colors of the graph represent symmetries present in the graph, i.e. two vertices are colored the same if they are \emph{automorphic}, {by which we mean that there is an automorphism $\phi$ of the graph such that these two vertices belong to the same orbit under $\phi$.} The black central vertex $v_{PR}$ with the most neighbors (Paul Revere) is connected to 248 of the 254 vertices.

The Gershgorin region $\Gamma(A)$ where $A$ is the adjacency matrix of the this social network is shown as the largest disk in Figure \ref{fig:revwar} (right). The region consist of the union of 254 concentric disks each centered at the origin the largest of which is the disk of radius 248, corresponding to the vertex $v_{PR}$. Hence, the spectral radius of $A$ is less than or equal to 248.

Using equitable decompositions we can decrease this Gershgorin estimate. The largest contributor to the size of the Gershgorin region is, in fact, the central vertex $v_{PR}$, since it has the highest degree of any vertex in the network. In order to reduce the associated Gershgorin disk, we look for an automorphism $\psi$ that permutes a subset of the vertices neighboring this hub, where our goal is to decompose the network's adjacency matrix $A$ over $\psi$. The issue we run into is that if $\psi$ fixes the vertex $v_{PR}$ then, as is shown in the proof of Theorem \ref{thm:Gersh}, the Gershgorin disk associated with $v_{PR}$ will not decrease in size as $A$ is decomposed over $\psi$. Consequently there will be no improvement in the Gershgorin region.

To actually improve this Gershgorin estimate, we note that finding the eigenvalues of a matrix is equivalent to finding the eigenvalues of its transpose. Thus, making Gershgorin regions from columns instead of rows, i.e. a \emph{transposed Gershgorin region} $\Gamma(M^T)$ of a matrix $M$ could potentially improve the Gershgorin region, even when the region is dominated by the fixed portion of the graph. One complication is that we are not guaranteed the Gershgorin region formed from the \emph{transpose} of the decomposed matrix columns will be contained in the Gershgorin region of the original matrix. Thus, this method of using a matrix' transposed Gershgorin region for gaining improved eigenvalue estimates may not be effective for all networks.

In this example we are, in fact, able to decrease the transposed Gershgorin region associated with the network by decomposing the matrix $A$ over certain automorphisms where points in nontrivial orbits are adjacent to $v_{PR}$. In fact, we are able to find a number of automorphisms that allow us to sequentially decompose the matrix $A$ such that at each step we gain an improve estimate of the network's eigenvalues.

This process results in a significant improvement in the original (transposed) Gershgorin estimate of the network's eigenvalues. This improvement is shown in Figure \ref{fig:revwar} (right), where each disk starting from the largest and moving inward represents the next transposed Gershgorin region after the corresponding equitable decomposition is performed on the network. The final, smallest region is the union of two disks one centered at 5 with radius 150 and one centered at 1 with radius 148, {which is roughly thirty-seven percent the size of the original Gershgorin region associated with the network}.
\end{example}

\section{Graphical Realization of Equitable Decompositions}\label{sec:6}

Carrying out an equitable decomposition is much more visual on a graph than on an associated matrix, and the (tedious) ordering of the labels is trivial when working with graphs instead of matrices. To illustrate the process of an equitable decomposition of a graph we introduce the notion of a folded graph. A folded graph can be thought of as a graph $G$ in which we have folded together all the vertices that are in the same orbit under an automorphism of $G$. The resulting graph can be used to generate a number of smaller graphs $G_{\phi},G_1,\dots,G_{k-1}$ that correspond to an equitable decomposition $M_{\phi}\oplus B_1\oplus\cdots\oplus B_{k-1}$ of a matrix $M$ associated with $G$.

To describe this procedure of folding a graph we first note that a graph $G$ can be either weighted or unweighted. If it is unweighted it is possible to weight the graph's edges by giving each edge unit weight. Under this convention any graph $G$ can be considered to be a weighted graph with weighted adjacency matrix $W=W(G)$.

For simplicity, we assume in this section that the eigenvalues associated with a graph $G$ are the eigenvalues of the graph's weighted adjacency matrix, which we denote by $\sigma(G)$. This can be done without loss in generality since any automorphism compatible matrix $M$ associated with $G$ is the weighted adjacency matrix of a graph $H$ where $Aut(G)=Aut(H)$.

By extending Theorem \ref{thm:2} to graphs, an equitable decomposition of a graph $G$ over a basic automorphism $\phi$ results in a number of smaller matrices $G_0,G_1\dots G_{k-1}$ where
\[\sigma(G)=\sigma(G_0)\cup\sigma(G_1)\dots\cup\sigma(G_{k-1})\]
where $W(G_i)=B_i$ is given by Equation \eqref{eq:B}. Here we describe how the graphs $G_{\phi},G_1\dots G_{k-1}$ can be generated from a single (folded) weighted graph $\mathcal{G}_\phi(m)$. To show how this is done we let $G=(V,E,w)$ denote the (weighted) graph $G$ with vertices $V=V(G)$ and  edges $E=E(G)$. The function $w:E\rightarrow\mathbb{C}$ gives the weight of each edge $(i,j)$ in $E$. The following steps allow one to equitably decompose a graph.

\begin{center}
\emph{Performing Equitable Graph Decompositions}
\end{center}

\begingroup\raggedright\leftskip=20pt\rightskip=20pt

\noindent\textbf{Step 1:} \emph{The basic automorphism $\phi$ partitions the graph's vertices into $V=V_1\cup\dots\cup V_{r}\cup U$, where $U$ is the union of the vertices of orbit size 1. Choose a semi-transversal $\mathcal{T}_0$ of the orbits of $\phi$, i.e. the index of one vertex from each set $V_i$. If this is not the last round, this semi-transversal must be chosen according to the method set out in Proposition \ref{prop:dd}}.


{\noindent\textbf{Step 2:} \emph{From the graph $G=(V,E,w)$ we construct the folded graph $\mathcal{G}_{\phi}(m)=(\mathcal{V}_m,\mathcal{E}_m,\nu_m)$ as follows. The vertices of the graph are labeled by $\mathcal{V}_m=\{\phi^m(i):i\in\mathcal{T}_0\cup\mathcal{U}\}$. The edge weights of $\mathcal{G}_{\phi}(m)$ are given by the formula}
\[
{\nu_m(\phi^m(i),\phi^m(j))=
\begin{cases}
\sum_{\ell=0}^{k-1}\omega^{\ell m}w(i,\phi^\ell(j)) \ &\text{if} \ j\in \mathcal{T}_0
\\
w(i,j) &\text{if } j \in U.
\end{cases}}
\]}

{\noindent\textbf{Step 3:} \emph{Next we generate the $k$ graphs $G_m=\mathcal{G}_{\phi}(m)$ for $m=0,\dots,k-1$ where $\mathcal{G}_{\phi}(0)=G_{\phi}$. If $m=0$ then $\mathcal{V}_0=\mathcal{T}_0\cup U$, otherwise $\mathcal{V}_m=\mathcal{T}_m$. (When drawing $\mathcal{G}_{\phi}(m)$ we draw the vertices in $\mathcal{U}$ as open circles to indicate that these vertices are only in the graph for $m=0$, see Figure \ref{fig:GED}).}}

\par\endgroup

\begin{figure}\label{fig:folded}
\begin{center}
 \begin{overpic}[scale=.2375]{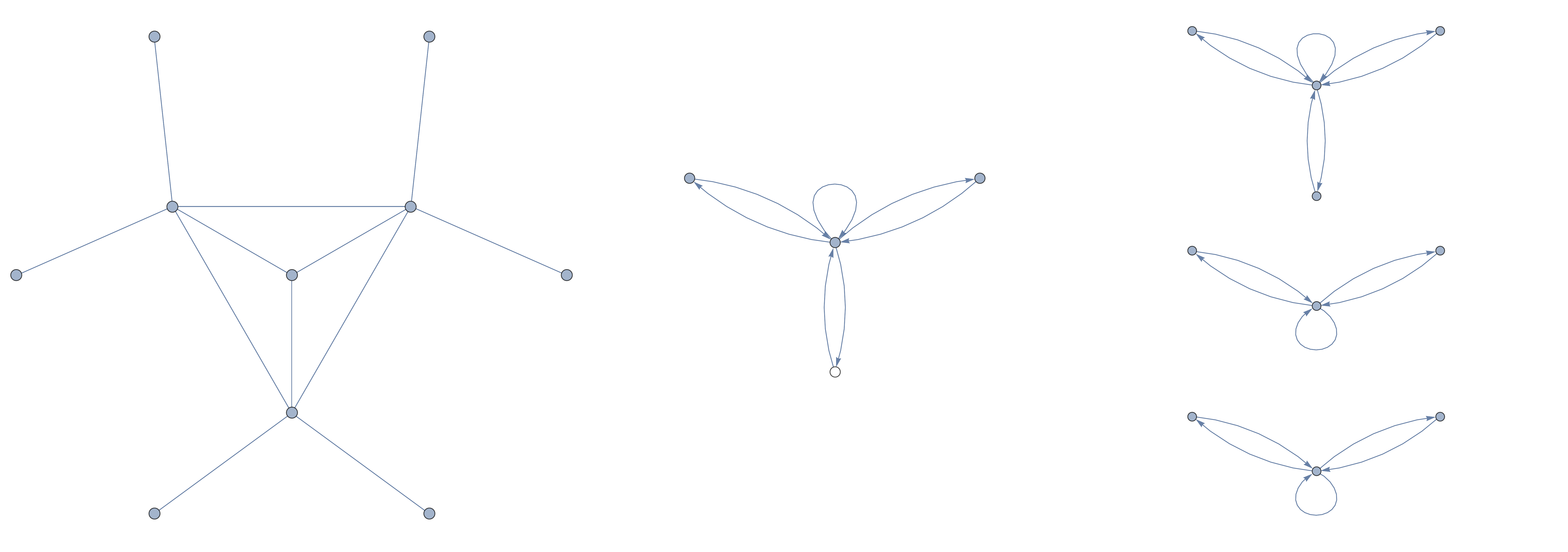}
 \put(18,-1){$G$}
  \put(17.25,16.5){\scriptsize$1$}
  \put(10.5,20.25){\scriptsize$5$}
  \put(0.4,16){\scriptsize$7$}
  \put(8.5,32.5){\scriptsize$6$}
  \put(26,20.25){\scriptsize$8$}
  \put(28.25,32.5){\scriptsize$10$}
  \put(36,16){\scriptsize$9$}
  \put(18.25,7){\scriptsize$2$}
  \put(28.25,2){\scriptsize$4$}
  \put(8.25,2){\scriptsize$3$}
  \put(18,4){$\mathcal{T}_0$}

  \put(7,2.5){ \fbox{\rule{1.3in}{0pt}\rule[-0.5ex]{0pt}{7.5ex}}}

 \put(50.5,5){$\mathcal{G}_{\psi}(m)$}
 \put(54.5,15){$1$}
 \put(50.75,15){$3$}
 \put(49.25,24){$\omega^m+\omega^{2m}$}

 \put(47,19){$1$}
 \put(47,23.5){$1$}
 \put(58.5,19){$1$}
 \put(58.5,23.5){$1$}

 \put(51.5,9.5){\scriptsize$\psi^m(1)$}
 \put(49.2,18.25){\scriptsize$\psi^m(2)$}
 \put(41.5,24.5){\scriptsize$\psi^m(3)$}
 \put(61.5,24.5){\scriptsize$\psi^m(4)$}

 \put(72,28){$G_{\psi}$}
 \put(83.5,33.5){$2$}
 \put(78.5,29){$1$}
 \put(78.5,33){$1$}
 \put(88,29){$1$}
 \put(88,33){$1$}
  \put(81.5,26){$3$}
 \put(85,26){$1$}

 \put(83.53,21){\scriptsize$1$}
 \put(74.75,33){\scriptsize$3$}
 \put(92.5,33){\scriptsize$4$}
 \put(82.5,28.5){\scriptsize$2$}

 \put(72,16){$G_{1}$}
 \put(78.5,15){$1$}
 \put(78.5,19){$1$}
 \put(88,15){$1$}
 \put(88,19){$1$}
 \put(82,11){$-1$}

 \put(74.75,19){\scriptsize$9$}
 \put(92.5,19){\scriptsize$10$}
 \put(83.55,16.5){\scriptsize$8$}

 \put(72,5){$G_{2}$}
 \put(78.5,4.5){$1$}
 \put(78.5,8.5){$1$}
 \put(88,4.5){$1$}
 \put(88,8.5){$1$}
 \put(82,0.5){$-1$}

 \put(74.75,8.5){\scriptsize$7$}
 \put(92.5,8.5){\scriptsize$6$}
 \put(83.55,6){\scriptsize$5$}

 \end{overpic}
\caption{The folded graph $\mathcal{G}_{\psi}(m)$ (center) of the unweighted graph $G$ (left) with basic automorphism $\psi=(2,8,5)(3,9,7)(4,10,6)$ created using the semi-transversal $\mathcal{T}_0=\{2,3,4\}$. The equitable decomposition $G_{\psi}$, $G_{1}$, $G_2$ of $G$ is shown (right) where $G_{\psi}=\mathcal{G}_{\psi}(0)$ and $G_m=\mathcal{G}_{\psi}(m)$ for $m=1,2$. Here $\omega=e^{2\pi i/3}$.}\label{fig:GED}
\end{center}
\end{figure}
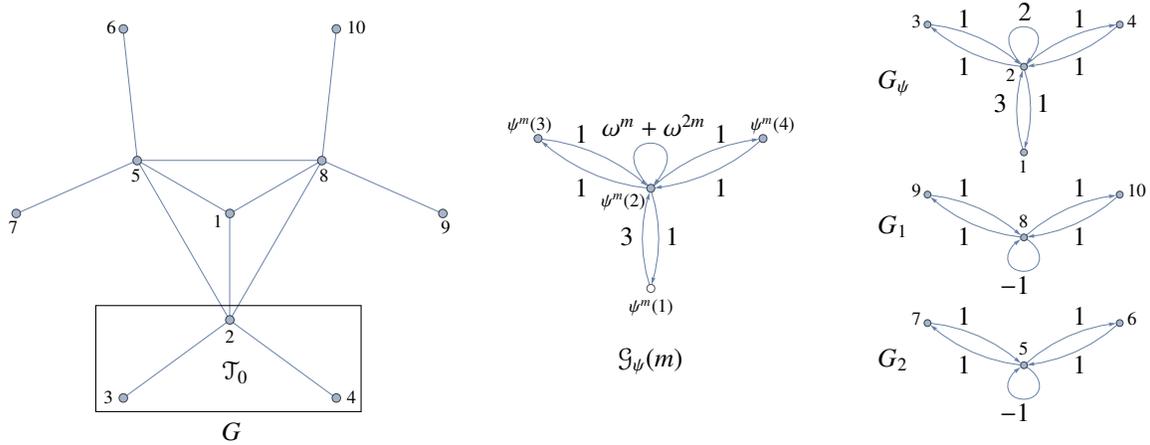

\vspace{0.1in}

If $\phi$ is a uniform automorphism of $G$ then $G$ can be equitably decomposed over $\phi$ where steps 1--3 are adjusted such that $\mathcal{U}=\emptyset$. If $\phi$ is separable then using the method described for decomposing a matrix with respect to a separable automorphism in Section \ref{sec:3} we can sequentially decompose $G$ over some sequence of basic automorphisms $\psi_0,\psi_1,\dots,\psi_h$ associated with $\phi$.

One of the main advantages of visualizing equitable decompositions in terms of a folded graph as opposed to the  matrix procedure prescribed in Section \ref{sec:3}, is that it makes the blocks of the decomposition immediately apparent.  For instance, recall that in Example \ref{ex:sequence}, at the end of Round 2 we were required to reorder the vertices (equivalent to relabeling the graph) in order to see the block diagonal structure. In the graphical approach presented in this section the connected subgraphs are the blocks resulting from the decomposition.

\begin{example}
Let $G$ be the graph originally introduced in Figure \ref{fig:ex1}, which is also shown in Figure \ref{fig:GED} (left), with the basic automorphism $\psi=(2,8,5)(3,9,7)(4,10,6)$. Following steps 1--3 given in this section, we first choose the semi-transversal $\mathcal{T}_0=\{2,3,4\}$ of $\psi$. The folded graph $\mathcal{G}_{\psi}(m)$ that results from this choice is shown in Figure \ref{fig:GED} (center). The graphs $G_{\psi}$, $G_1$, $G_2$, which together form an equitable decomposition of the graph $G$ over $\psi$ are also shown in the same figure (right). One can check that $\sigma(G)=\sigma(G_{\psi})\cup\sigma(G_1)\cup\sigma(G_2)$ since the matrices $W(G_{\psi})=A(0)_{\psi_0}$, $W(G_1)=B(0)_1$, and $W(G_2)=B(0)_2$ given in Equation \eqref{eq:7}.
\end{example}

\section{Conclusion}\label{sec:7}

The purpose of this paper is to extend the theory of equitable decomposition as well as to introduce a number of its applications. {The theory of equitable decompositions, first presented in \cite{BFW}, describes how an automorphism compatible matrix $M$ can be decomposed over any \emph{uniform} or \emph{basic} automorphism of an associated graph. In this paper we extend this result by providing a method for equitably decomposing $M$ over any separable automorphism $\phi$ by converting any such automorphism into a sequence of basic automorphisms. As in \cite{BFW}, this decomposition results in a number of smaller matrices $M_{\phi},B_1,\dots,B_h$ whose collective eigenvalues are the eigenvalues of the original matrix $M$ (see Theorem \ref{thm:2} and Proposition \ref{thm:3}). Importantly, this decomposition relies only on a knowledge of the automorphism $\phi$ and requires no information regarding any spectral properties of the graph.}

As previously mentioned, if $\phi$ is not a separable automorphism then some power of this automorphism will be separable. Therefore, if any automorphism of a graph is known it is possible to use this automorphism or some power of this automorphism to equitably decompose an associated matrix M. Additionally, the algorithm we give for equitably decomposing a graph depends both on the order of the prime factorization of the automorphism's order as well as some choices surrounding the semi-transversals. {An open question is} whether making these choices differently {will} result in a fundamentally different decomposition (i.e. not just a simple reordering of resulting block matrices) or if, up to reordering, the decomposition is, in fact, unique.

{Beyond preserving the eigenvalues of a matrix, we also show as a direct application of this theory that the eigenvectors of the matrix $M$ are fundamentally related to the eigenvectors of the matrices $M_{\phi},B_1,\dots,B_k$ and give an explicit formula for their construction (see Theorem \ref{eigvec}).} This theorem also extends to generalized eigenvectors in the case where the original matrix does not have a full set of eigenvectors. {In this way the concept of an equitable decomposition can be applied not only to matrices but also to their eigenvectors.}

A theme throughout this paper, {and in particular in the applications introduced here}, is that information regarding {spectral} properties of a graph {(network)} can be deduced from knowledge of the graph's {(network's)} local symmetries. One example, shown in Proposition \ref{lem:3}, is that {if $M$ is nonnegative and irreducible then the divisor matrix $M_{\phi}$ associated with the automorphism $\phi$ has the same spectral radius as $M$. Since $M_{\phi}$ is smaller than $M$, and potentially much smaller if the graph (network) is highly symmetric, then $M_{\phi}$ could be a useful tool in determining this spectral radius. It is an open question whether this result can be extended to matrices with negative or complex-valued entries.}

Another application mentioned here is concerned with the way in which Gershgorin regions are affected by equitable decompositions. Here, we prove that the Gershgorin region of an equitable decomposition is contained in the Gershgorin region of the original matrix (see Theorem \ref{thm:Gersh}). Thus, such decompositions can be used to potentially reduce the size of the resulting Gershgorin region. {In particular,} Example \ref{ex:revwar} demonstrates this by {significantly reducing} the size of the Gershgorin region {associated with a large network via a series of decompositions}.

{Lastly, we introduce the analogue of an equitable decomposition for graphs. An equitable graph decomposition, which is built around the notion of a folded graph,  has the advantage that it is more visual than an equitable decomposition on a matrix. Moreover, performing a graph decomposition does not} require any reordering of the graph's vertices as opposed to the row and column reordering that is required in an equitable decomposition of a matrix.

\section*{Acknowledgments}
{The authors would like to thank Wayne Barrett for his many helpful comments and suggestions regarding this paper. We would also like to thank the anonymous referee for their valuable feedback. The work of A. Francis and B. Webb was partially supported by the DoD grant HDTRA1-15- 0049.}

\bibliography{references}{}
\bibliographystyle{plain} 

\end{document}